\documentclass[11pt]{article}
\usepackage[margin=1in]{geometry}
\usepackage{amsmath,amsthm,amssymb}
\usepackage{environ, tabularx}
\usepackage{tikz-cd,graphicx}
\usepackage{float}                      \usepackage{indentfirst}
\usepackage{mathtools}
\usepackage{hyperref}
\usepackage{amsmath}
\usepackage{dirtytalk}
\usepackage{pifont}
\usepackage{enumitem}
\usepackage{cleveref}
\usepackage{relsize}
\usepackage{graphicx}
\usepackage{subcaption}
\allowdisplaybreaks

\usetikzlibrary{arrows}
\usetikzlibrary{backgrounds}
\usetikzlibrary{calc}
\usetikzlibrary{cd}
\usetikzlibrary{decorations.pathmorphing}
\usetikzlibrary{fit}
\usetikzlibrary{petri}
\usetikzlibrary{positioning}
\usetikzlibrary{trees}

\DeclareMathOperator{\R}{\mathbb{R}}
\DeclareMathOperator{\C}{\mathbb{C}}
\DeclareMathOperator{\E}{\mathbb{E}}
\DeclareMathOperator{\Tr}{Tr}

\DeclareMathOperator{\T}{\mathcal{T}}

\DeclareMathOperator{\st}{\mathbb{~|~}}

\newcommand\one{\mathbf{1}}
\newcommand{\parexp}[2][]{{#1}^{(#2)}} 

\newtheorem{theorem}{Theorem}[section]

\newtheorem{lemma}[theorem]{Lemma}
\newtheorem{corollary}[theorem]{Corollary}
\newtheorem{proposition}[theorem]{Proposition}
\newtheorem{conjecture}[theorem]{Conjecture}
\theoremstyle{remark}

\newtheorem{example}{Example}

\title{Spectral properties of random graphs\\ with fixed equitable partitions}
\author{Matthew Crawford, David Marchette, William Maxwell, Samuel Mendelson}
\date{Naval Surface Warfare Center, Dahlgren Division}

\begin{document}

\maketitle

\begin{abstract}
    We define a graph to be $S$-regular if it contains an equitable partition given by a matrix $S$. These graphs are generalizations of both regular and bipartite, biregular graphs. An $S$-regular matrix is defined then as a matrix on an $S$-regular graph consistent with the graph's equitable partition. In this paper we derive the limiting spectral density for large, random $S$-regular matrices as well as limiting functions of certain statistics for their eigenvector coordinates as a function of eigenvalue. These limiting functions are defined in terms of spectral measures on $S$-regular trees. In general, these spectral measures do not have a closed-form expression; however, we provide a defining system of polynomials for them. Finally, we explore eigenvalue bounds of $S$-regular graph, proving an expander mixing lemma, Alon-Bopana bound, and other eigenvalue inequalities in terms of the eigenvalues of the matrix $S$.
\end{abstract}

\section{Introduction}
Let $G$ be an undirected, simple graph, that is a graph with no loops or multiple edges. We say $G$ is \emph{$S$-regular} if $G$ has an equitable partition given by the matrix $S$. These graphs are generalizations of both regular and bipartite biregular graphs; the former having an equitable partition given by just the set of vertices and the latter having an equitable partition given by the bipartition. We define an $S$-regular matrix as a matrix on an $S$-regular graph which is consistent with the graph's equitable partition. The goal of this paper is to understand various spectral properties of families of $S$-regular matrices.

In \cite{McKay1981}, McKay derived the expected spectrum of large, random $d$-regular graphs. These results were tied to the counting of closed walks. Godsil and Mohar further refined the relationship between spectra and closed walks in \cite{Godsil1988}, showing how the expected spectral distribution for certain classes of graphs could be approximated by a limiting graph for those classes. Both sets of results required sequences of large graphs to converge to a limiting graph. Independently, both Wormald \cite{WORMALD1981168} and Bollobas \cite{Bollobs1980} showed  the short cycle distribution in large, uniform-random $d$-regular graphs are almost independent Poisson variables in the length of the cycles; importantly showing large, random, $d$-regular graphs are locally tree-like and enabling the results of \cite{McKay1981} and \cite{Godsil1988} for $d$-regular graphs.

More recently, similar techniques have been applied to the study of eigenvectors of $d$-regular graphs; see, for example \cite{brooks2013non,dumitriu2012sparse,bauerschmidt2017local}. These authors show that, with high probability, the eigenvectors of the adjacency matrix of a large, random, $d$-regular graph are \emph{delocalized}. Informally, the mass of an eigenvector is not centered on a small number of coordinates.

In this paper, we extend these ideas for $d$-regular graphs to $S$-regular matrices. We show large, random, $S$-regular graphs are locally tree-like, thus having the $S$-regular tree as a limiting graph in the vein of \cite{Godsil1988}. Consequently, the spectral measures of the $S$-regular tree can be used to approximate the expected spectral measures of large, random, $S$-regular matrices. Further, we show these measures can be used to approximate the limiting cumulative distribution function for the variance of normalized-eigenvector coordinates. We conjecture that, not only are the eigenvectors of these graphs delocalized, but with high probability have a variance approximated by the density of the corresponding cumulative distribution functions.

While the limiting spectral measure of large, random, $d$-regular graphs has a closed form expression, in general, this is not the case for $S$-regular graphs. Instead, we derive a defining system of polynomials of the various spectral measures of an $S$-regular tree. Using these polynomials, we provide numerical approximations for the expected statistics of large $S$-regular graphs and compare them to empirical data.

Finally, We show that several well-known eigenvalue bounds for the adjacency matrices of regular graphs have natural generalizations to $S$-regular graphs. We prove generalizations of the expander mixing lemma and the Alon-Boppana bound. We also prove a diameter bound and a bound on walks avoiding a fixed subgraph. Eigenvalue bounds for regular graphs often rely on the largest eigenvalue and corresponding eigenvector of the adjacency matrix. For $d$-regular graphs, these are just $\lambda = d$ and $x=\mathbf{1}$, the all-ones vector (in terms of $[d]$-regular graphs, these are the eigenvalue and eigenvector of $[d]$). We generalize these $d$-regular eigenvalue bounds by, instead, considering the eigenvalues of $S$ and their corresponding eigenvectors.

The study of $S$-regular matrices on trees has recently been explored by Avni, Breuer, and Simon \cite{avni2020periodic}, though under the name of \emph{periodic Jacobi operators}. Here, the authors consider $S$-regular trees and operators for which $S$ is the adjacency matrix $A$ of an undirected simple graph without leaves. They show that the Stieltjes transform of spectral measures on these trees are algebraic and use this fact to deduce properties of the spectrum of these operators. We generalize this statement to general $S$-regular graphs. We note, if $A$ is $S$-regular for $S\neq A$, then the $S$-regular tree and the $A$-regular tree are the same graph. However, while all $S$-regular matrices are also $A$-regular operators, the converse is not true.

The remainder of this paper is organized as follows: in \Cref{sec:prelim} we provide definitions and background that will be used throughout the rest of the paper; in \Cref{sec:spec} we characterize expected spectral properties of large $S$-regular graphs; and in \Cref{sec:eig} we present eigenvalue bounds for families of $S$-regular graphs.

\section{Preliminaries}
\label{sec:prelim}
Throughout this paper $G=(V,E)$ will be a simple, connected, undirected graph with $|V|=n$ and adjacency matrix $A_G$, or simply $A$ when the context is clear. 
Because $A$ is symmetric, there exists an eigenvalue decomposition $A = \Phi^T \Lambda \Phi$ where the columns of $\Phi$ form an orthonormal basis of eigenvectors for $A$. The standard basis vectors for a vector space will be denoted $e_i$ and the $i$-th component of a vector $\phi$ will be denoted $\phi(i)$. For $B, C \subseteq V$ we denote the set of edges between $B$ and $C$ by $E(B, C)$. By $N(v)$ we denote the neighborhood of $v$ and by $N_B(v)$ we denote $N(v) \cap B$.

We call a partition $V_1,\dots,V_k$ of $V$ an \emph{equitable partition} if there is a matrix $S$ (called the quotient matrix or degree refinement matrix of the equitable partition) such that each $v \in V_i$ has exactly $s_{ij}$ neighbors in $V_j$. In notation, $|N_{V_j}(v)| = s_{ij}$, or $v$'s degree in $V_j$ is $s_{ij}$. If $G$ has an equitable partition given by some matrix $S$ we say that $G$ is \emph{$S$-regular}. All $S$-regular graphs will be assumed to have an equitable partition $V_1,\dots,V_k$, where $S\in M_k$ and $V_i\neq \emptyset$. Note that $S$-regular is a generalization of \emph{$d$-regular}, a graph such that every vertex has degree equal to $d$ by simply viewing the scalar $d$ as a $1 \times 1$ matrix $[d]$. A graph $G$ is always $A_G$-regular.

We will use the following notation $|V_i| = n_i$, $B_i = V_i \cap B$ for $B \subseteq V$, and $b_i = |B_i|/n_i$. Note that $\sum_{i=1}^k n_i = n$ and $\sum_{i=1}^k |B_i| = |B|$. Because $G$ is assumed to be undirected, we have the following \emph{balance equations} $s_{ij}n_i = s_{ji}n_j$ for all $1\leq i,j\leq k$ (both quantities equal $|E(V_i, V_j)|$).  If $S$ is fixed, $n_i/n$ is constant for all $S$-regular graphs. We call the subsets $V_1,\dots,V_k$ the \emph{cells} of the equitable partition and $\tau \colon V \rightarrow \{1,\dots,k\}$ is the \emph{cell function} which maps a vertex $v \in V_i$ to the index of its corresponding cell $i$. The indicator function of a set will be denoted $\mathbf{1}_B$. In an abuse of notation, the indicator vector will also be denoted $\mathbf{1}_J\in \R^n$ where $J\subset [n]$. The all-ones vector is then $\mathbf{1}=\mathbf{1}_{[n]}$. 

The underlying graph of a matrix $T$ is a directed graph $G$ with $|V|=\text{dim}(T)$ and an edge between $v_i$ and $v_j$ if and only if $T_{ij}\neq 0$. If $T$ is a matrix with underlying graph $G$, a \emph{$T$-weighted walk of length $\ell$} is a walk of length $\ell$ on $G$ along with a product of all the weights of the edges traversed. In this paper, we consider only Hermitian matrices, so we may assume all underlying graphs are undirected. We denote the number of closed walks of length $\ell$ starting and ending at vertex $v$ as $\parexp[W]{n}(v)$ and the sum of weights of all of those $T$-weighted walks as $\omega_T^{(\ell)}(v)$.

We define an $S$-regular matrix to be a Hermitian matrix $T$ with underlying graph $G$ satisfying the following properties:

\begin{enumerate}
    \item $G$, without loops and undirected, is $S$-regular,
    \item there exists $b\in\R^{k}$ such that, if $u\in V_i$, $T_{uu}=b(i)$,
    \item there exists $F\in M_k(\C)$ such that $F_{ij}\neq0$ for all $1\leq i,j\leq k$ and, if $u\in V_i$ and $v\in V_j$, $T_{uv} = F_{ij}$.
\end{enumerate}

We will refer to the vector $b$ as the vertex weights and the matrix $F$ as the edge weights of $T$. If $G$ is an $S$-regular graph, its adjacency matrix, combinatorial Laplacian, and normalized Laplacian are all examples of $S$-regular matrices. The spectrum of an operator $T$ will be denoted $\sigma(T)$. The spectral density of a finite-dimensional operator $T$ will be denoted $$\rho(T)(B) = \frac{1}{\dim(T)}\sum\limits_{\lambda\in\sigma(T)}m(\lambda)\mathbf{1}_\lambda(B),$$ where $m(\lambda)$ is the multiplicity of $\lambda$ and $B\subset \C$.

The following result serves to provide some intuition about the connections between $S$-regular matrices and the quotient matrix $S$.

\begin{proposition}\label{prop:s-eigen}
For an $S$-regular matrix $T$ with vertex weights $b$ and edge weights $F$,
\begin{enumerate}
    \item the matrix $(S\circ F)+\text{diag}(b)$ is diagonalizable where $S\circ F$ is the component-wise product of $S$ and $F$ and $\text{diag}(b)$ is the diagonal matrix with entries $b$,
    \item if $\lambda$ is an eigenvalue of $(S\circ F)+\text{diag}(b)$ with eigenvector $\psi\in\R^k$, then $\lambda$ is an eigenvalue of $T$ with eigenvector $\overline{\psi}\in\R^{|V|}$ where $\overline{\psi}(v) = \psi({\tau(v)})$,
    \item if $\phi$ is an eigenvector of $T$ whose corresponding eigenvalue is not an eigenvalue of $(S\circ F)+\text{diag}(b)$, then $\sum\limits_{v\in V_i}\phi({v}) = 0$ and $\sum\limits_{v\in V}\phi({v}) = 0$.
\end{enumerate}
\end{proposition}

\begin{proof}
    \begin{enumerate}
        \item Let $N$ be the diagonal matrix with $N_{ii} = n_i$. Then $$N^{1/2} [(S\circ F)+\text{diag}(b)] N^{-1/2} = N^{1/2} S N^{-1/2}\circ F + N^{1/2} \text{diag}(b) N^{-1/2}.$$ By $S$-regular balance equations, $N^{1/2} S N^{-1/2}$ is symmetric and so $N^{1/2} S N^{-1/2}\circ F + N^{1/2} \text{diag}(b) N^{-1/2}$ is Hermitian and is thus diagonalizable. Hence, $(S\circ F)+\text{diag}(b)$ is also diagonalizable since the two matrices are similar.

        \item Let $v\in V_i$. Because $T$ is an $S$-regular matrix and the vector $\overline{\psi}$ defined above is constant across the cells $V_i$:
        \begin{align*}
            \left(T\overline{\psi}\right)({v}) & = T_{vv}\overline{\psi}({v}) + \sum_{j=1}^k\sum_{\substack{u\sim v \\ u\in V_j}} T_{uv}\overline{\psi}({u}) \\
            &= b({\tau(v)}){\psi}({\tau(v)}) + \sum_{j=1}^k\sum_{\substack{u\sim v \\ u\in V_j}} F_{\tau(u)\tau(v)} \psi({\tau(u)}) \\
            & = b({i})\psi({i}) + \sum_{j=1}^k F_{ij}s_{ij}\psi({j}) = \lambda \psi({i})  = \lambda \psi({\tau(v)}) = \lambda \overline{\psi}({v}).
        \end{align*}
        Therefore, $\overline{\psi}$ is an eigenvector of $T$ with eigenvalue $\lambda$.

        \item Because $(S\circ F)+\text{diag}(b)$ is diagonalizable, there exists a basis for $\R^k$ of eigenvectors of $(S\circ F)+\text{diag}(b)$, $\{\psi_i\}_{i=1}^k$. Therefore, $e_j$ can be written as a linear combination $e_j=\sum_{i=1}^k c_i 
        \psi_i$. If $\overline{\psi}_i$ are eigenvectors of $T$ corresponding to $\psi_i$, then $\mathbf{1}_{V_j} = \sum_{i=1}^k c_i \overline{\psi}_i$. Suppose $\phi$ is an eigenvector of $T$ whose corresponding eigenvalue is not an eigenvalue of $(S\circ F)+\text{diag}(b)$. Then, because $T$ is Hermitian, $\langle \phi, \overline{\psi}_i\rangle = 0$ for all $1\leq i\leq k$ and so $\langle \phi, \mathbf{1}_{V_j}\rangle=0$ for all $1\leq j\leq k$. Consequently, $\langle \phi,\mathbf{1}\rangle=0$.
    \end{enumerate}
\end{proof}

\Cref{prop:s-eigen} splits the eigenvectors of an $S$-regular matrix into two types: those corresponding to $(S\circ F)+\text{diag}(b)$ and those with mean 0 across all partition cells. When $w=\mathbf{0}$ and $F=\mathbf{1}$, we obtain the following corollary for the adjacency matrix of an $S$-regular graph:

\begin{corollary}\label{cor:s-eig}
For an $S$-regular graph $G$ with adjacency matrix $A$
\begin{enumerate}
    \item the matrix $S$ is diagonalizable
    \item if $\lambda$ is an eigenvalue of $S$ with eigenvector $\psi\in\R^k$, then $\lambda$ is an eigenvalue of $A$ with eigenvector $\overline{\psi}\in\R^{|V|}$ where $\overline{\psi}({v}) = \psi({\tau(v)})$,
    \item if $\phi$ is an eigenvector of $A$ whose corresponding eigenvalue is not an eigenvalue of $S$, then $\sum\limits_{v\in V_i}\phi({v}) = 0$ and $\sum\limits_{v\in V}\phi({v}) = 0$,
    \item let $J_m$ be the matrix whose $uv$-th entry is $\frac{(S^m)_{\tau(u)\tau(v)}}{n_{\tau(v)}}$. If $\lambda_1,\dots,\lambda_k$ be eigenvalues of $S$ with corresponding eigenvectors $\psi_1,\dots,\psi_k$, then,
    \[
   J_m = \sum\limits_{i=1}^k \lambda_i^m\frac{\overline{\psi}_i\overline{\psi}_i^T}{||\overline{\psi}_i||}.
    \]
\end{enumerate}
\end{corollary}

\begin{proof}
    \begin{itemize}
        \item[1-3.] These statements are immediate consequences of \Cref{prop:s-eigen} when $F=0$ and $b=0$.
        \item[4.]  We have 
        \[
        \left(J_m\overline{\psi}_i\right)(u) = \sum_{j=1}^k n_j\frac{(S^m)_{\tau(u)j}}{n_j}\psi(j) = \sum_{j=1}^k {(S^m)_{\tau(u)j}}\psi(j) = \lambda_i^m\psi_i(\tau(u)) = \lambda_i^m\overline{\psi_i}(u).
        \] Let $\phi$ be orthogonal to $\overline{\psi_i}$ for $1\leq i \leq k$. Then, by (3), $J_m\phi = 0$. Thus $J_m$ has eigen-decomposition $\sum\limits_{i=1}^k \lambda_i^m\frac{\overline{\psi}_i\overline{\psi}_i^T}{||\overline{\psi}_i||} $.
    \end{itemize}
\end{proof}

Define $\mathcal{G}(S)$ to be the family of $S$-regular graphs. We form a partial order on the quotient matrices of equitable partitions by set inclusion of the corresponding family of graphs; $S\preceq S'$ if and only if $G(S)\subseteq G(S')$. This partially ordered set forms a collection of disconnected lattices \cite{Prof_Edward_R_Scheinerman2013-12-20}, with each lattice corresponding to a shared universal cover and having a minimum, coarsest equitable partition. We denote the $S$-regular tree $\T_S$, which is the universal cover of all finite $S$-regular graphs. If $\parexp[G]{1}$ and $\parexp[G]{2}$ are two $S$-regular graphs, then it is always possible to create two covering maps $\pi_1:\T_S\rightarrow\parexp[G]{1}$ and $\pi_2:\T_S\rightarrow\parexp[G]{2}$ such that $\pi_1^{-1}(V_i)=\pi_2^{-1}(V_i)$ for all $1\leq i\leq k$. Letting $V_i^\ast :=\pi_1^{-1}(V_i)$, the set $\{V_i^\ast\}_{i=1}^k$ is an equitable partition of $\T_S$. For this reason, shared indices of multiple $S$-regular graphs will always correspond to the same pre-image in $\T_S$. If $T$ is an $S$-regular matrix,  $T$ can be extended to an $S$-regular matrix $T^\ast$ with the same edge and vertex weights as $T$ whose simple and undirected underlying graph is $\T_S$. As shorthand, we will often refer to $T$-weighted walks in this case as $T$-weighted walks on $\T_S$.

An \emph{irreducible matrix} is a matrix whose corresponding directed graph is strongly connected. For any connected $S$-regular graph $G$, the  matrix $S$ will be irreducible whose entries satisfy the balance equations $n_i s_{ij} = n_j s_{ji}$. To conclude the section, we show that for any irreducible $S$ whose entries satisfy a set of balance equations, $\mathcal{G}(S)\neq \emptyset$.

\begin{proposition}
For any irreducible $k \times k$ matrix $S$ with non-negative integer entries, if there exists a non-trivial solution to the equations $n_is_{ij}=n_js_{ji}$, there exists a connected $S$-regular graph.
\end{proposition}
\begin{proof}
Suppose there exist $\{n_i\}_{i=1}^k$ such that $n_is_{ij}=n_js_{ji}$. Then, for all $\alpha$, $\alpha n_is_{ij}=\alpha n_js_{ji}$. Without loss of generality, we assume $n_i$ is a positive integer such that $n_i>s_{ii}$, $n_i\geq s_{ji}$ for all $1\leq j\leq k$, and $n_is_{ii}$ is even. 

Let $V$ be a set with $|V|=\sum_{i=1}^k n_i$ with partition $V_1,\dots,V_k$ such that $|V_i|=n_i$. An $S$-regular graph is a set of $s_{ii}$-regular graphs connected in a bipartite $(s_{ij},s_{ji})$-biregular way, so we will construct an $S$-regular graph by building all of its pieces individually. First, because $n_is_{ii}$ is even and $n_i>s_{ii}$, we are able to construct an $s_{ii}$-regular graph on each $V_i$. Next, because $n_is_{ij}=n_js_{ji}$ and $n_i\geq s_{ji}$ for all $i$ and $j$, we can construct a bipartite, $(s_{ij},s_{ji})$-biregular graph with bipartition $(V_i,V_j)$ with the set of $n_i/s_{ji}$ disjoint, complete bipartite graphs $K_{s_{ji},s_{ij}}$. Putting all these pieces together, we have an $S$-regular graph. If the graph is disconnected, each connected component will be $S$-regular as well. And so, there exists a connected, $S$-regular graph.
\end{proof}

\section{Large random \texorpdfstring{$S$}{TEXT}-regular graphs}
\label{sec:spec}

\subsection{Configuration model}
The \emph{configuration model} is a method for sampling uniform random $d$-regular graphs originating with Bollob{\'{a}}s~\cite{Bollobs1980}. Much work has been done on models for random $d$-regular graphs and we recommend the survey of Wormald for an overview~\cite{Wormald1999}. Notably, the configuration model has been expanded to bipartite biregular graphs~\cite{Lubetzky2010, Brito2021}. Since an $S$-regular graph decomposes into independent regular and bipartite biregular graphs it follows that we can use the configuration model to sample uniform random $S$-regular graphs by sampling each piece of the graph independently.

We give a brief description of the configuration model for $S$-regular graphs. For each cell in the equitable partition we have a set $V_i = \{v_{i,1}, \dots, v_{i, n_i}\}$ and for each pair of cells we have the set of \emph{half-edges} $\vec{E}_{i,j} = \{ (v_{i,k}, \ell) \mid 1 \leq k \leq n_i, 1 \leq \ell \leq s_{ij} \}$. By the balance equation we have $|\vec{E}_{i,j}| = |\vec{E}_{j,i}|$.
For $i \leq j$ let $\pi_{ij}$ be a uniform random permutation on $\{1,\dots, s_{ij}n_i\}$ and set $\pi_{ji} = \pi_{ij}^{-1}$. For $v_{i'} \in V_i$ and $v_{j'} \in V_j$ we place an edge between $v_{i'}$ and $v_{j'}$ if and only if $\pi_{ij}((i' - 1)s_{ij} + \ell) = \pi_{ji}((j' - 1)s_{ji} + \ell')$ for some $\ell,\ell'$. Note that this process may produce non-simple graphs, however by preconditioning on the event that we sample a simple graph we obtain a uniform random graph. See~\cite{Bollobs1980} for more details.

\subsection{Locally tree-like}
For our results about the expected spectrum of an $S$-regular graph to hold it is necessary to show that large, random $S$-regular graphs are locally tree-like. In particular, we need to show that for a large enough $S$-regular graph within any ball of fixed radius the expected number of cycles is close to zero. This allows us to prove statements about large, random $S$-regular graphs by working with the universal cover, the $S$-regular tree. Our proof computes the expected number of cycles within a ball of fixed radius by revealing the edges of a breadth-first search one at a time. This technique was used for similar results for $d$-regular and bipartite biregular graphs~\cite{Lubetzky2010, Brito2021}.

\begin{lemma}\label{lem:expected_cycles}
Let $G$ be an $S$-regular graph on $n$ vertices uniformly sampled from the configuration model. For any vertex $v \in V$ let $X_v$ be the number of cycles in the ball $B_r(v)$, then $\E[X_v] = \Theta(1/n)$.
\end{lemma}
\begin{proof}
Perform a breadth-first search on $v$ using the configuration model to reveal one edge at a time. We perform the breadth-first search until the entire ball $B_r(v)$ is revealed. Each edge is revealed by matching a pair of half-edges.
Let $A_{\ell,k}$ denote the event that the $k^{\text{th}}$ edge revealed at depth $\ell$ creates a cycle. This can only occur if the edges endpoint is one of the $k-1$ vertices already revealed at depth $\ell$. Without loss of generality assume that the edge connects cell $i$ with cell $j$. There are at least $\sum_{i=1}^k d_i n_i - 2d_{\text{max}}^{i+1}$ unmatched edges remaining in $G$. Hence, \[\Pr[A_{\ell,k}] = \frac{n_j}{n} \times \frac{(k-1)(d_{\text{max}} - 1)}{\sum_{i=1}^k d_i n_i - 2d_{\text{max}}^{i+1}} = \Theta \left( \frac{1}{n} \right).\]
The probability that there are $k$ cycles in $B_r(v)$ is asymptotically equivalent to a binomial distribution on at most $d_{\text{max}}^{r+1}$ trials, so $\Pr[X_v = k] = \Theta(1/n^k)$. The expected number of cycles is given by \[\E[X_v] = \sum_{i=1}^{d_{\text{max}}^{r+1}} \Theta\left(\frac{1}{n^i} \right) = \Theta \left( \frac{1}{n} \right)\] which completes the proof.
\end{proof}

With \Cref{lem:expected_cycles} we can show that for large, random $S$-regular graphs the number of closed walks of length $\ell$, in some sense, converges almost surely to the number of closed walks of length $\ell$ on the $S$-regular tree. Consequently, the sum of weights of weighted, closed walks will also converge to the sum of weights of weighted, closed walks on the $S$-regular tree.

\begin{theorem}\label{thm:weighted_walks}
    Let $\{T_n\}_{n=1}^\infty$ be a sequence of $S$-regular matrices with shared vertex weights and edge weights and $\{v_n\}_{n=1}^\infty$ be a sequence of vertices of their underlying graphs with $v_n\in \parexp[V_i]{n}$ with $\text{dim}(T_n)\rightarrow\infty$. If $T^\ast$ is the corresponding $S$-regular matrix on $\T_S$ with self-loops determined by vertex weights and $v^\ast\in V_i^\ast$, then, 
    \[
    \omega_{T_n}^{(\ell)}(v_n) \xrightarrow[]{\text{a.s.}} \omega_{T^\ast}^{(\ell)}(v^\ast)
    \]
\end{theorem}

\begin{proof}
Let $\ell>0$ and consider the ball $B_n = B_{\lceil \frac{l}{2} \rceil}(v_n)$. All closed walks of length $\ell$ will be contained in this ball. We partition these closed walks into two types: acyclic backtracking walks with self loops, and the excess walks. The acyclic backtracking walks with self loops, denoted $\mathrm{abt}_\ell(v_n)$, are closed walks that contain no cycles, except self loops and the entire walk, such that each time a non-loop edge is traversed, it must be traversed again in the opposite direction at some point. The excess walks, denoted $\mathrm{ex}_\ell(v_n)$, contain all other closed walks. Acyclic backtracking walks with loops starting at $v_n$ are in one-to-one correspondence with closed walks starting at $v^\ast$ the $S$-regular tree with self loops where dictated by vertex weights.

The quantity $\mathrm{ex}_\ell(v_n)$ is dependent on the number of cycles in $B_n$. For any fixed cycle $\gamma$ in $B_n$, the maximum number of closed walks containing $\gamma$ is bounded above by a constant since they are contained in a ball of fixed radius. By \Cref{lem:expected_cycles} $\E[\mathrm{ex}_\ell(v_n)] = O\left(\frac{1}{|\parexp[V]{n}|}\right)$ since it is bounded above by the constant times the expected number of cycles in $B_n$. Therefore, as $n\rightarrow \infty$,
\[
\E[\parexp[W]{\ell}(v_n)] = \E[\mathrm{abt}_\ell(v_n)] + \E[\mathrm{ex}_\ell(v_n)] = \mathrm{abt}_\ell(v_n) + \E[\mathrm{ex}_\ell(v_n)] \rightarrow \mathrm{abt}_\ell(v_n) = \parexp[W]{\ell}(v^\ast).
\]

Because $\parexp[W]{\ell}(v^\ast)\leq\parexp[W]{\ell}(v_n)$, we have $\Pr\left[\lim\limits_{n\rightarrow\infty}\parexp[W]{\ell}(v_n)=\parexp[W]{\ell}(v^\ast)\right]=1$. Because the number of closed walks of length $\ell$ converges almost surely and the edge and vertex weights are shared across all $T_n$ and $T^\ast$, we have $\Pr\left[\lim\limits_{n\rightarrow\infty}\parexp[\omega]{\ell}(v_n)=\parexp[\omega]{\ell}(v^\ast)\right]=1$.
\end{proof}

\subsection{Expected Spectral Properties of Large \texorpdfstring{$S$}{TEXT}-regular Matrices}

A spectral measure is a projection-valued measure on a measurable space. Because we only consider Hermitian operators, we will only consider spectral measures on compact subsets of $\R$ with the usual Borel sets. We state the relevant theory for our purposes; but, for a thorough introduction to the area, we suggest \cite{halmos2017introduction}. Spectral measures with compact support in $\R$ are in one-to-one correspondence with Hermitian operators given by
\[
E \Leftrightarrow \int_{\R} \lambda dE(\lambda),
\] 
where $E$ is a spectral measure with compact support in $\R$. When the support of $E$ is discrete, this correspondence is just a consequence of the spectral theorem for finite Hermitian matrices. The corresponding operator $A$ is finite dimensional and the support of $E$ is equal to $\sigma(A)$, and, we have
\[
E(B) = \sum_{\lambda\in\sigma(A)} \Pi_\lambda\mathbf{1}_\lambda(B) \Leftrightarrow A = \sum_{\lambda\in\sigma(A)} \lambda\Pi_\lambda,
\] where $B\subset\R$ is a Borel set and $\Pi_\lambda$ is the projection onto the eigenspace of $\lambda$.

We can induce real-valued measures from spectral measures using inner products. In particular, the measure $\mu_j(B) = \langle E(B)e_j,e_j\rangle$ is a probability measure. If a sequence of these measures are induced from a sequence of $S$-regular matrices, the limiting measure is found using an operator on the $S$-regular tree.

\begin{theorem}\label{prop:vertex_dist}
Let $\{T_n\}_{n=1}^\infty$ be a sequence of $S$-regular matrices with $\text{dim}(T_n)\rightarrow \infty$ as $n\rightarrow\infty$ with shared vertex and edge weights and $E_n$ their corresponding spectral measures. Let $\{v_n\}_{n=1}^\infty$ be a sequence of vertices of their underlying graphs with $v_n\in V_i^{(n)}$, $T^\ast$ the corresponding $S$-regular matrix on $\T_S$ with self-loops determined by vertex weights, $E^\ast$ its corresponding spectral measure, and $v^\ast\in V_i^\ast$. If $\mu_{v_n} = \langle E_n e_{v_n},e_{v_n}\rangle$ and $\mu_i = \langle E^\ast e_{v^\ast},e_{v^\ast}\rangle$, then, almost surely, $\mu_{v_n} \rightarrow \mu_i$ in distribution. 
\end{theorem}

\begin{proof}
    Explicitly, we compute the $j$-th moment of $\mu_{v_n}$ as
    \[
    \int_{\R} \lambda^j d\mu_{v_n} = \sum_{\lambda\in\sigma(T_n)} \lambda^j\langle\Pi_\lambda e_{v_n},e_{v_n}\rangle = \left\langle\left(\sum_{\lambda\in\sigma(T_n)} \lambda^j\Pi_\lambda\right) e_{v_n},e_{v_n}\right\rangle = \langle T_n^j e_{v_n},e_{v_n}\rangle = \omega_{T_n}^{(j)}(v_n).
    \]
    Because $\mu_i$ is a spectral measure and $\displaystyle{T^\ast = \int\limits_{\R} \lambda d\mu_i}$, we have $\displaystyle{(T^\ast)^j = \int\limits_{\R} \lambda^j d\mu_i}$. And so the $j$-th moment of $\mu_i$ is $\omega_{\T^\ast}^{(j)}(v^\ast)$. By \Cref{thm:weighted_walks}, $\omega_{T_n}^{(j)}(v_n)\xrightarrow{a.s.} \omega_{\T^\ast}^{(j)}(v^\ast)$, and so, almost surely, $\mu_{v_n}$ converges to $\mu_i$ in distribution.
\end{proof}

Summing over all vertices in the underlying graph of $T_n$, we obtain the following corollary:

\begin{corollary}\label{cor:density}
Under the hypotheses of \Cref{prop:vertex_dist}and if $v_i\in V_i^\ast$ for $1\leq i\leq k$ and $\mu_i = \langle E^\ast v_i,v_i\rangle$, then, almost surely, $\rho(T_n)\rightarrow \sum\limits_{i=1}^k c_i \mu_i$ in distribution.
\end{corollary}

\begin{proof}
    The spectral density of $T_n$ is the trace of the spectral measure $E_n$ over the dimension of $T_n$:
    \[
    \rho(T_n)(B) = \frac{1}{\dim T_n}\sum_{\lambda\in\sigma(T_n)} m(\lambda)\mathbf{1}_\lambda (B) = \frac{1}{|V^{(n)}|} \sum_{\lambda\in\sigma(T_n)} \Tr(\Pi_\lambda)\mathbf{1}_\lambda (B) = \frac{1}{|V^{(n)}|} \Tr(E_n(B))
    \] since the trace of a projection is the dimension of its range. The trace of $E_n$ is the sum over the measures $\mu_v$ for $v\in V^{(n)}$. By \Cref{prop:vertex_dist}, if $v\in V_i^{(n)}$, then $\mu_v$ converges to $\mu_i$ in distribution. For each $i$, the sum includes $|V_i^{(n)}|$ many measures converging to $\mu_i$ in distribution. Thus, $\Tr(E_n)$ almost surely converges to $\sum\limits_{i=1}^k c_i \mu_i$ in distribution.
\end{proof}

Informally, \Cref{prop:vertex_dist} gives the limiting cumulative distribution for the square of a fixed coordinate across a basis of normalized eigenvectors of a finite-dimensional, Hermitian matrix. If $\{\phi_{i}\}_{i=1}^n$ is an orthonormal set of eigenvectors with eigenvalues $\lambda_1\leq\dots\leq\lambda_n$ of $T$ with spectral measure $E$, we have:
\[
\mu_{v}(B) = \langle E(B) e_{v},e_{v}\rangle = \sum_{\lambda\in\sigma(T)} \left(\Pi_\lambda\right)_{vv}\mathbf{1}_{\lambda}(B) = \sum_{i=1}^n \left(\phi_i{(v)}\right)^2\mathbf{1}_{\lambda_i}(B).
\]

Because of the weak convergence of \Cref{prop:vertex_dist}, for large $S$-regular matrix with $v\in V_i$, we have the cumulative distribution function approximated as
\[
\sum_{\lambda_i\leq b} \left(\phi_i{(v)}\right)^2 \approx \int_{-\infty}^b d\mu_i(\lambda).
\] 

While we should not expect the corresponding probability density functions to look similar in any way, evidence suggests that the probability density functions of averages of these measures are close to the density function of $\mu_i$. Coupled with the fact that most eigenvectors have mean 0 over partition cells of the underlying graph, we present the following conjecture:

\begin{conjecture}
    Let $T$ be a large $S$-regular matrix and $T^\ast$ the corresponding $S$-regular matrix on $\T_S$. If $U\subset V_i$ is a large enough subset of the $i$-th partition cell of the underlying graph of $T$, $\phi$ is a normalized eigenvector of $T$ with eigenvalue $\lambda$, and $\mu_i$ is defined as in \Cref{prop:vertex_dist}, then, with high probability
    \begin{itemize}
        \item[1.] $\sum\limits_{v\in U} \phi{(u)} \approx 0$,
        \item[2.] $\displaystyle{\frac{|V|}{|U|}\sum\limits_{v\in U} \left(\phi{(u)}\right)^2 \approx \frac{\mu_i(\lambda)}{\sum_j^k c_j\mu_j(\lambda)}}$
    \end{itemize} where $c_i = |V_i|/|V|$.
\end{conjecture}

The denominator on the right hand side of (2) appears because $\left(\phi{(u)}\right)^2$ is a density of $\mu_u$ with respect to $\rho(T)$ whose limiting distribution is $\sum_j^k c_j\mu_j(\lambda)$ by \Cref{cor:density}.

\subsection{Spectral Measures on the \texorpdfstring{$S$}{TEXT}-Regular Tree}

When $S$ defines either a regular or biregular, bipartite $S$-regular graph, it is possible to find the measures on $\T_S$ defined in \Cref{cor:density}\cite{Mohar1989}. However, for general $S$, no closed formula is known. Typically, to compute these measures, one first finds a a generating function $W(y)$ for the sequence of closed walks starting at a particular vertex and then computes the inverse Stieltjes transform of $\frac{1}{z}W\left(\frac{1}{z}\right)$. Again, for the regular and biregular tree, these walk generating functions can be computed explicitly. In this section, we develop a system of polynomials for which these walk-generating functions are a solution, which, in general, have no closed-form expression.

In the $S$-regular tree any two vertices in the same cell of the equitable partition given by $S$ will have the same number of closed walks rooted at them. This is because for any $v_1, v_2 \in V_i$ there is an automorphism on the $S$-regular tree mapping $v_1$ to $v_2$. For this reason, and to save on notation, we introduce the terms $W_i^{(\ell)}$ and $\omega_i^{(\ell)}$to denote the number of closed walks and sum of weighted-closed walks, respectively, in the $S$-regular tree of length $\ell$ rooted at a vertex $V_i$. 

We begin this section with a sequence of technical lemmas describing the $T$-weighted walks on the $S$-regular tree.

\begin{lemma}
\label{lem:spec1}
    Let $T$ be an $S$-regular operator with vertex weights $b$ and edge weights $F$. The sum of $T$-weighted, closed walks of length $l$ on $\T_S$ satisfy the following recurrence relations:
    \begin{enumerate}
        \item $\omega_{i}^{(0)},\omega_{ij}^{(0)}=1$
        \item $\omega_i^{(1)}=b{(i)}$, $\omega_{ij}^{(1)}=b{(j)}$
        \item $\displaystyle{\omega_i^{(l)} = b{(i)}^l + \sum_{j=1}^k s_{ij}F_{ij}\left(\sum_{r+s+t~\leq~ l-2}  b(i)^r\omega_i^{(s)} \omega_{ij}^{(t)}\right)}$, for $l>1$\\
        \item $\displaystyle{\omega_{ij}^{(l)} = b(j)^l + \sum_{m=1}^k {(s_{jm}-\delta_{im})F_{jm}}\left(\sum_{r+s+t~\\ \leq~l-2}  b(j)^r\omega_{ij}^{(s)} \omega_{jm}^{(t)}\right)}$, for $l>1$
    \end{enumerate}
\end{lemma}

\begin{proof}
    (1) There is exactly one closed walk of length 0.\\

    (2) Let $u\in V_i$. There is exactly one $T$-weighted, closed walk at $u$ and it has weight $b(i)$. The same argument holds for $\omega_{ij}^{(1)}$.\\

    (3) Let $u\in V_i$. First, there a walk of length $l$ that consists of $l$ loops at $u$. This walk has weight $b(i)^l$. Next, we consider walks that eventually leave $u$. These walks start with some number of loops less than $l-2$ at $u$ because we leave $u$ and eventually must come back. These steps have weight $b(i)$. Suppose $r$ many loops have been traversed. After these loops, we sum over the weighted walks of length $l>0$ starting at $u$ that return after exactly $t$ steps. We choose a $v\in V_j$ to step towards of which there are $s_{ij}$ choices. Because $\T_S$ is a tree and we are considering closed walks, we will traverse the edge from $u$ to $v$ and back which contributes a weight of $F_{ij}$. Once at this vertex, we complete a walk of length $t$ without returning to $u$. The sum of these weighted walks is counted precisely by $\omega_{ij}^{(t)}$. We then return to $u$ and have a walk of length $s = l-r-t-2$ remaining. The sum of these weighted walks is counted by $\omega_i^{(s)}$. The number of closed walks starting at $u$ is then counted by summing over $j$ and $r+s+t\leq l-2$ and we obtain \[
    \displaystyle{\omega_i^{(l)} = b(i)^l + \sum_{j=1}^k s_{ij}F_{ij}\left(\sum_{r+s+t~\leq~ l-2}  b(i)^r\omega_i^{(s)} \omega_{ij}^{(t)}\right)}.
    \]

    (4) Let $v\in V_j$ with $s_{ij}>0$ and $u\in V_i$ adjacent to $v$. The sum of $T$-weighted closed walks starting at $v$ with the edge to $u$ removed is counted in the same way as above; however, there is one less vertex in $V_i$ adjacent to $v$. We obtain \[\omega_{ij}^{(l)} = b(j)^l + \sum_{m=1}^k {(s_{jm}-\delta_{im})F_{jm}}\left(\sum_{r+s+t~\leq~l-2}  b(j)^r\omega_{ij}^{(s)} \omega_{jm}^{(t)}\right).    
    \]
\end{proof}

\begin{lemma}
\label{lem:spec2}
    Let $\omega_i$ and $\omega_{ij}$ be the generating functions for $\omega_i^{(l)}$ and $\omega_{ij}^{(l)}$ respectively. Then \begin{align*}
    0 =& 1 - (1-b(i)y)\omega_{i} + y^2\omega_{i}\sum_{m=1}^k s_{im}F_{im}\omega_{im} \text{ for }1\leq i\leq k \text{ and }\\
    0 =& 1 - (1-b(j)y)\omega_{ij} + y^2\omega_{ij}\sum_{m=1}^k (s_{jm}-\delta_{im})F_{jm}\omega_{jm}\text{ for } 1\leq i,j\leq k\text{ when }s_{ij}>0.
    \end{align*}
\end{lemma}

\begin{proof}
    Using \Cref{lem:spec1}, we have
    \begin{align*}
    \omega_i & = \sum_{l=0}^\infty \omega_i^{(l)}y^l\\
    & = 1 + b(i)y + \sum_{l=2}^\infty \left[b(i)^l + \sum_{j=1}^k s_{ij}F_{ij}\left(\sum_{r+s+t~\leq~ l-2}  b(i)^r\omega_i^{(s)} \omega_{ij}^{(t)}\right)\right]y^{l}\\
    & = \sum_{l=0}^\infty b(i)^ly^l + y^2\sum_{j=1}^k s_{ij}F_{ij}\sum_{l=2}^\infty\left(\sum_{r+s+t~\leq~ l-2}  b(i)^r\omega_i^{(s)} \omega_{ij}^{(t)}\right)y^{l-2}\\
    & = \frac{1}{1-b(i)y} + y^2\sum_{j=1}^k s_{ij}F_{ij}\left(\sum_{l=0}^\infty b(i)^ly^l\right)\left(\sum_{l=0}^\infty \omega_i^{(l)}y^{l}\right)\left(\sum_{l=0}^\infty \omega_{ij}^{(l)}y^{l}\right)\\
    & = \frac{1}{1-b(i)y} + y^2\frac{1}{1-b(i)y}\omega_{i}\sum_{j=1}^k s_{ij}F_{ij}\omega_{ij}.
    \end{align*}
    And $\displaystyle{0 = 1 - (1-b(i)y)\omega_{i} + y^2\omega_{i}\sum_{m=1}^k s_{im}F_{ij}\omega_{im}}$. 
    
    Similarly, $\displaystyle{0 = 1 - (1-b(j)y)\omega_{ij} + y^2\omega_{ij}\sum_{m=1}^k (s_{jm}-\delta_{im})F_{jm}\omega_{jm}}$ when $s_{ij}>0$.
\end{proof}

We now have a system of polynomials for which the $T$-weighted walk generating functions on the $S$-regular tree $\T_S$ are solutions. In fact, the $T$-weighted walk generating functions are in algebraic and the unique solution in a neighborhood of $y=0$. The following theorem is a generalization of \cite[Theorem~6.6]{avni2020periodic} and the proof is attributed to those authors. It is included here for completeness. 

\begin{theorem}
\label{thm:spec3}
    Let $X = \{X_{ij}\st 1\leq i,j\leq k, s_{ij}>0\}\cup\{X_{i}\st 1\leq i\leq k\}$ and $P$ be the system of polynomials
    \begin{align*}
        p_{i}(X) &=1 - (1-b(i) y)X_{i} + y^2\sum_{m=1}^k F_{im}s_{im}X_{i}X_{im}\text{ for }  1\leq i\leq k\\
        p_{ij}(X) &=1 - (1 - b(j) y)X_{ij} + y^2\sum_{m=1}^k F_{jm} (s_{jm}-\delta_{im})X_{ij}X_{jm}\text{ for } 1\leq i,j\leq k\text{ such that }s_{ij}>0.
    \end{align*}
    If $T$ is an $S$-regular matrix on $\T_S$, the $T$-weighted, closed walk generating functions are the unique solution to $P(X)=0$ in a neighborhood around $y=0$. Moreover, each $\omega_{ij}$ and $\omega_i$ are algebraic.
\end{theorem}

\begin{proof}
    When $y=0$, the above polynomials have the unique solution, $X_{ij}=X_i=1$ for all $i,j$. The Jacobian of this system is given by \[\left.\left[\frac{\partial p_\alpha}{\partial X_\beta}\right]\right|_{y=0} = -I\] and is invertible. By the implicit function theorem, there exists a unique solution to this system in some neighborhood around $y=0$, $X_{ij}=1$, and $X_i=1$. We know the $T$-weighted, closed walk generating functions satisfy the above polynomials and $\omega_{ij}(0)=\omega_i(0)=1$. Thus, the $T$-weighted, closed walk generating functions are the unique solution at $y=0$.

    Let $K$ be the field of rational, complex functions in $y$ and $\omega = \{\omega_{ij}\st 1\leq i,j\leq k, s_{ij}>0\}\cup\{\omega_{i}\st 1\leq i\leq k\}$.By \cite[Proposition~VIII.5.3]{lang2005algebra}, because $P(\omega)=0$ and the Jacobian of $P$ is invertible, $K(\omega)$ is a separable, algebraic extension of $K$.
\end{proof}

\begin{corollary}
\label{cor:spec4}
    Let $T$ be an $S$-regular matrix on $\T_S$ with associated spectral measure $E$. The Stieltjes transform of the spectral measure $\mu_{u}=\langle E e_u, e_u\rangle$ is algebraic.
\end{corollary}

\begin{proof}
    Suppose $u\in V_i$. The Stieltjes transform of $\mu_{u}$ is given by 
    \[
    R(z) = \int_{\lambda\in\sigma(T)} \frac{d\mu_u(\lambda)}{z - \lambda} = \frac{1}{z} \int_{\lambda\in\sigma(T)} \frac{d\mu_u(\lambda)}{1 - \frac{\lambda}{z}} = \frac{1}{z} \sum_{i=0}^\infty \frac{\omega^{(i)}(u)}{z^i} = \frac{1}{z}\omega_i\left(\frac{1}{z}\right)
    \] for $z\neq 0$ and $z\notin \sigma(T)$.
\end{proof}

All empirical evidence suggests that not only are the $T$-weighted walk generating functions algebraic, but the system of polynomials defined in \Cref{thm:spec3} has only finitely many solutions over the field of complex, rational functions and is thus zero dimensional. In particular this would imply that for each walk generating function $\omega_i$ and $\omega_{ij}$ there exists a \emph{Gr\"obner basis} containing a bivariate polynomial for which $\omega_i$ and $\omega_{ij}$ is a solution. We put forth the following conjecture:

\begin{conjecture}
\label{conj:spec6}
    Let $K$ be the field of algebraic functions over $\C$ and $P$ be defined as in \Cref{thm:spec3}. The ideal generated by $P$ is zero dimensional.
\end{conjecture}

We finish this section with a few examples. In each example we compare the limiting distribution of matrices on $\T_S$ with empirical data. For each example, we generate 20 instances of an $S$-regular matrix of dimension 2500. Our theoretical distributions are computed using the polynomials found in \Cref{thm:spec3} and Gr\"obner basis methods to compute bivariate polynomials for each walk generating function. Finally, we numerically approximate the solution to these polynomials and their inverse Stieltjes transform.

\begin{example}
In \Cref{fig:adj_matrix} and \Cref{fig:lap_mat} we consider the adjacency matrix and normalized Laplacian, respectively, of $S$-regular graphs where $S = {\begin{bmatrix}
        14 & 2\\
        2 & 2
\end{bmatrix}}.$ In \Cref{sub:adj_spec} and \Cref{sub:lap_spec}, we plot a histogram of the density of eigenvalues of our sampled graphs. The red curve represents the limiting distribution of the spectral density found in \Cref{cor:density}. In \Cref{sub:adj_1,sub:adj_2,sub:lap_1,sub:lap_2}, each white dot represents the sum of squared normalized-eigenvector coordinate taken over a partition cell plotted against the corresponding eigenvalue. The red curve represents the spectral measure on $\T_S$ of that cell divided by the limiting distribution of the spectral density scaled by $|V_i|/|V|$. 

\begin{figure}
 \centering
 \begin{subfigure}[b]{0.3\textwidth}
     \centering
     \includegraphics[width=\textwidth]{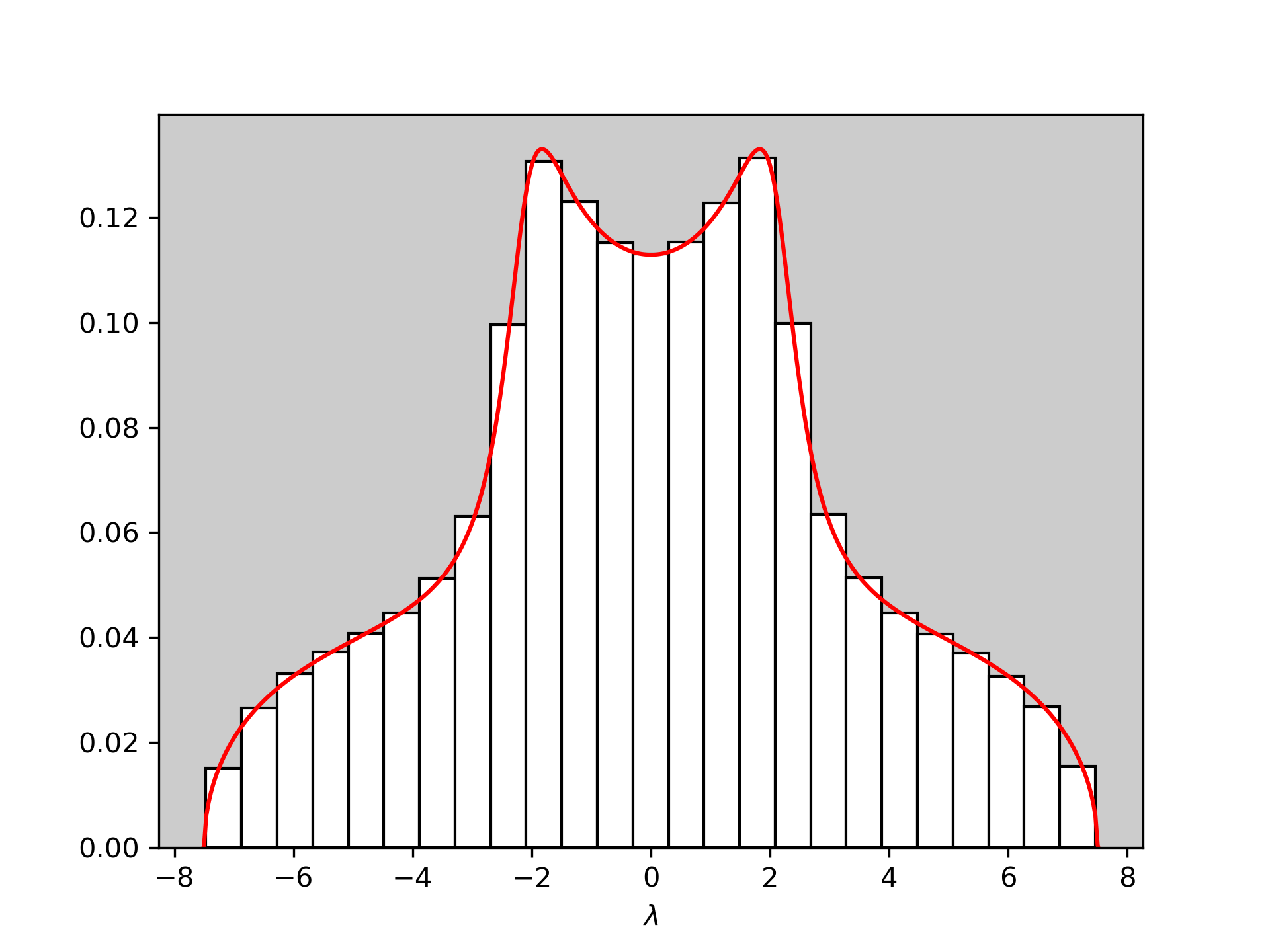}
     \caption{Spectral Density}
     \label{sub:adj_spec}
 \end{subfigure}
 \hfill
 \begin{subfigure}[b]{0.3\textwidth}
     \centering
     \includegraphics[width=\textwidth]{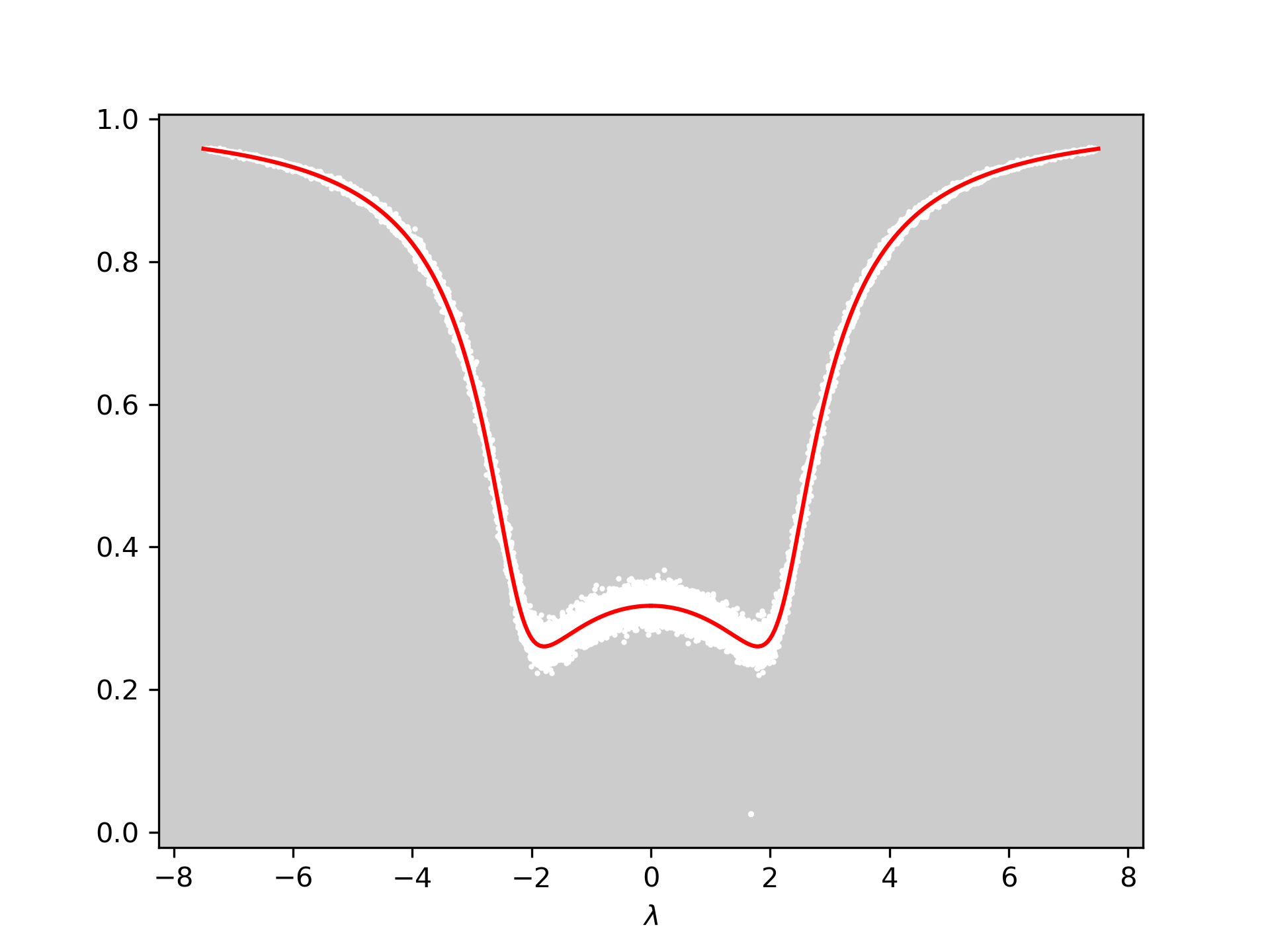}
     \caption{Sum of Square-Coord. $V_1$}
     \label{sub:adj_1}
 \end{subfigure}
 \hfill
 \begin{subfigure}[b]{0.3\textwidth}
     \centering
     \includegraphics[width=\textwidth]{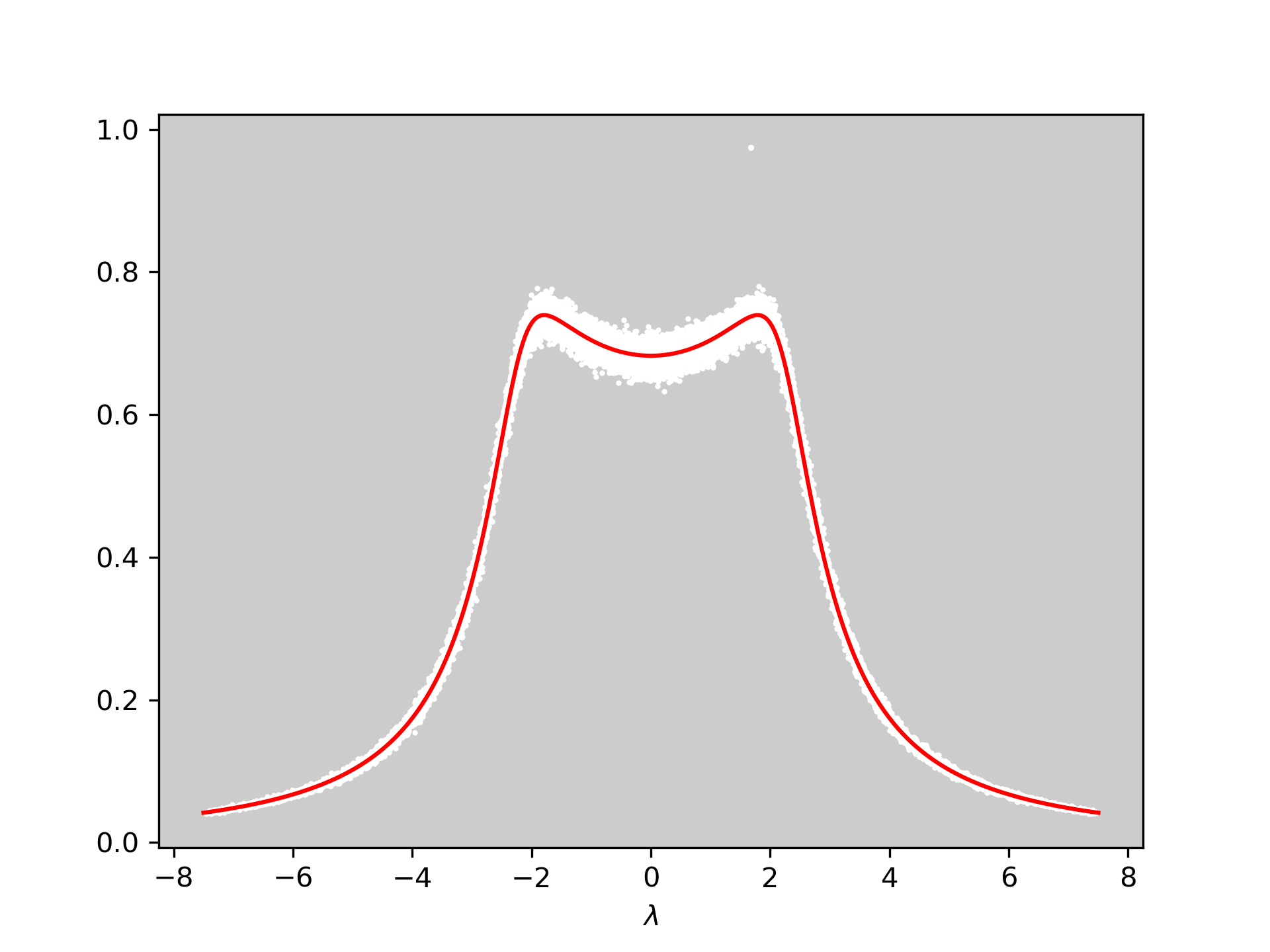}
     \caption{Sum of Square-Coord. $V_2$}
     \label{sub:adj_2}
 \end{subfigure}
    \caption{Adjacency Matrix}
    \label{fig:adj_matrix}
\end{figure}

\begin{figure}
 \centering
 \begin{subfigure}[b]{0.3\textwidth}
     \centering
     \includegraphics[width=\textwidth]{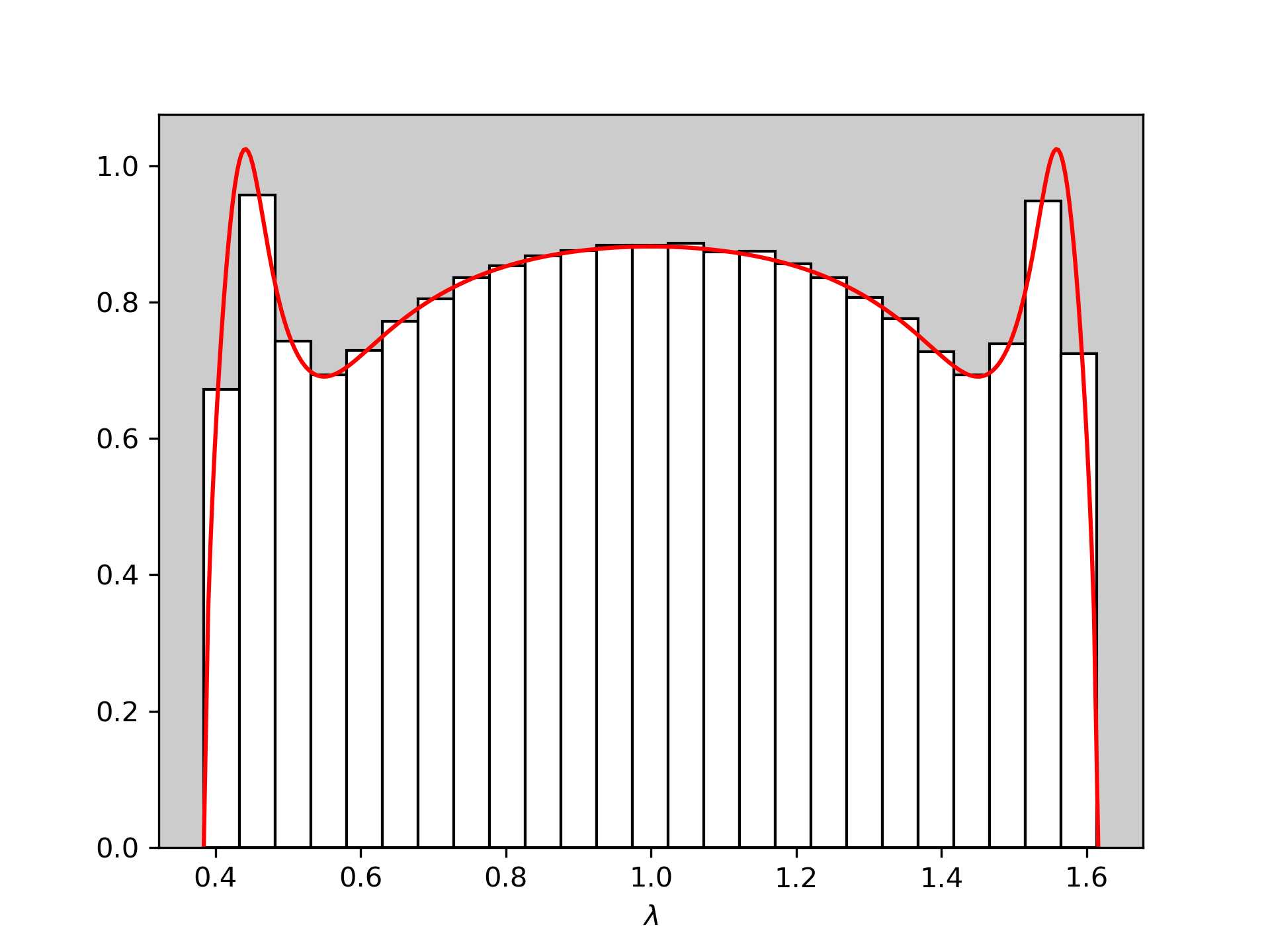}
     \caption{Spectral Density}
     \label{sub:lap_spec}
 \end{subfigure}
 \hfill
 \begin{subfigure}[b]{0.3\textwidth}
     \centering
     \includegraphics[width=\textwidth]{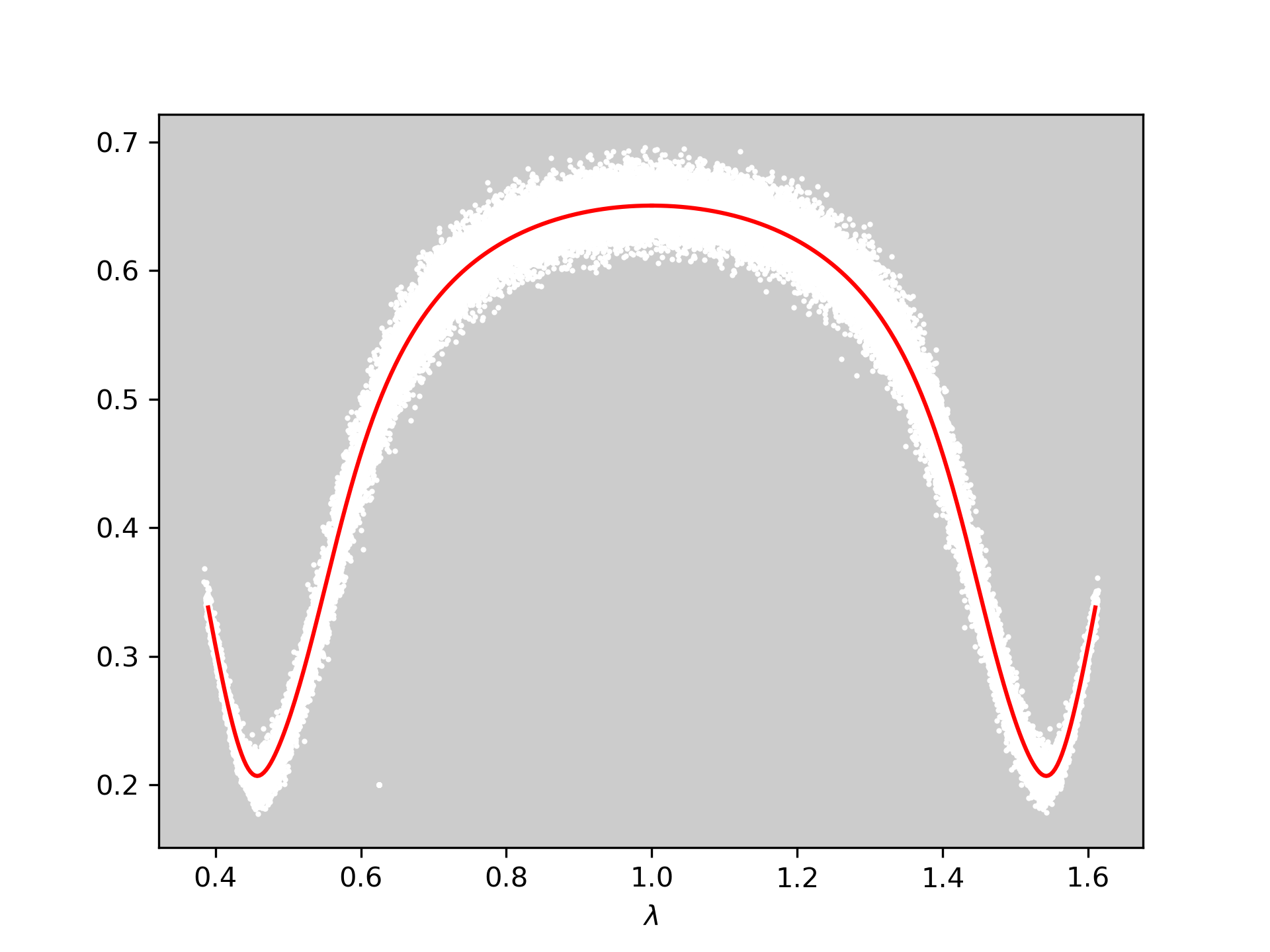}
     \caption{Sum of Square-Coord. $V_1$}
     \label{sub:lap_1}
 \end{subfigure}
 \hfill
 \begin{subfigure}[b]{0.3\textwidth}
     \centering
     \includegraphics[width=\textwidth]{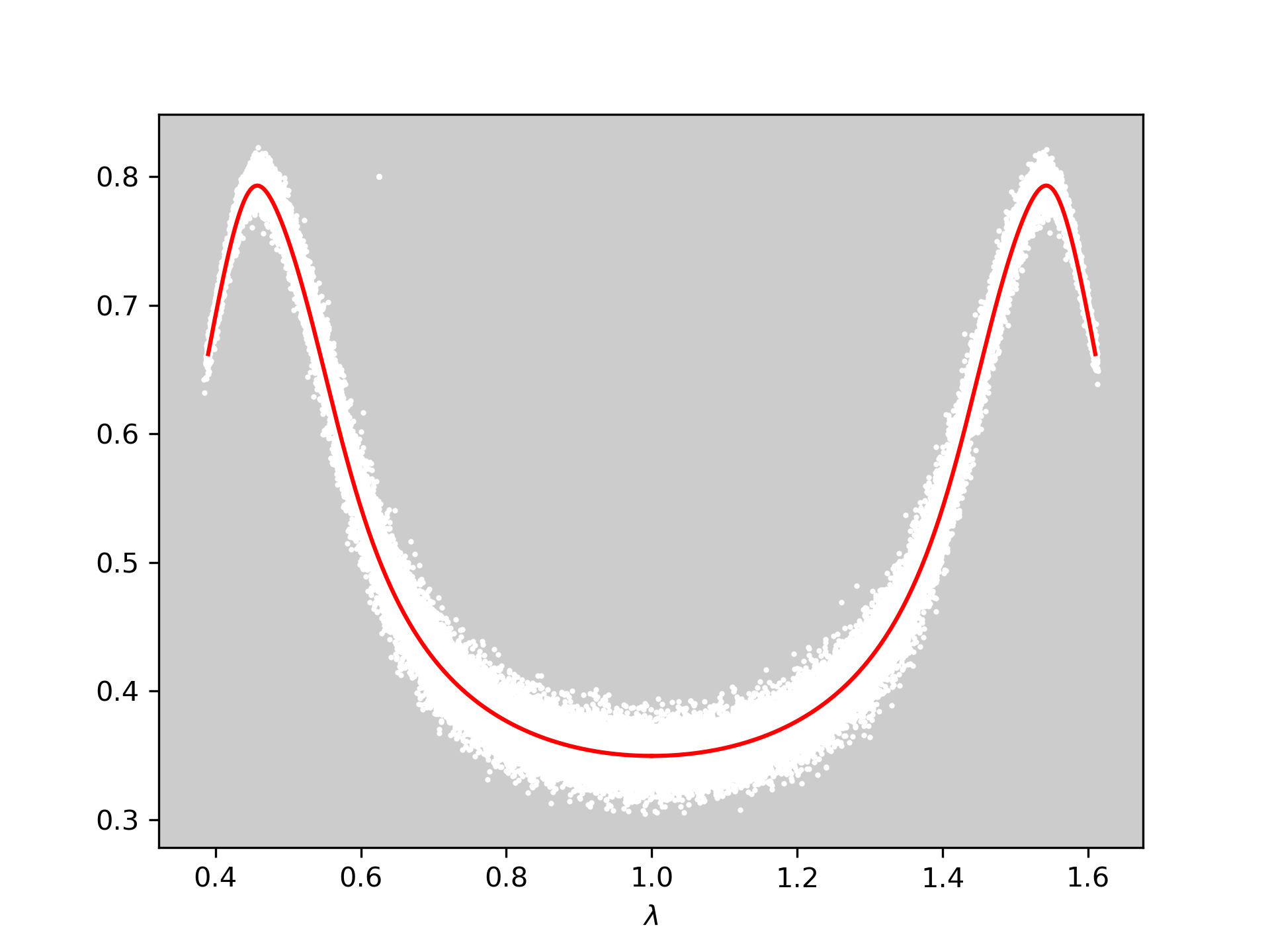}
     \caption{Sum of Square-Coord. $V_2$}
     \label{sub:lap_2}
 \end{subfigure}
    \caption{Normalized Laplacian}
    \label{fig:lap_mat}
\end{figure}

\Cref{fig:House} contains similar plots, this time for \[
S = \begin{bmatrix}
    0 & 1 & 1 & 0 & 0\\
    1 & 0 & 1 & 1 & 0\\
    1 & 1 & 0 & 0 & 1\\
    0 & 1 & 0 & 0 & 1\\
    0 & 0 & 1 & 1 & 0
\end{bmatrix},
\] whose graphical representation is the house graph. We consider the adjacency matrix of these graphs. The equitable partition given by $S$ contains a coarser equitable partition with matrix \[
T = \begin{bmatrix}
    0 & 2 & 0\\
    1 & 1 & 1\\
    0 & 1 & 1
\end{bmatrix},
\] and so there are only three unique spectral measures instead of five on $\T_S$. Interestingly, because of the structure of the house graph, there is a bijection between closed cycles starting in $V_2$ and closed cycles starting in $V_3$ (as well as a bijection between closed cycles starting in $V_4$ and those starting in $V_5$). Because of this, the sums over $V_2$ and $V_3$ are identical (as are the sums over $V_4$ and $V_5$).

\begin{figure}
    \centering
    \begin{subfigure}[b]{.3\textwidth}
        \centering
        \includegraphics[width=\textwidth]{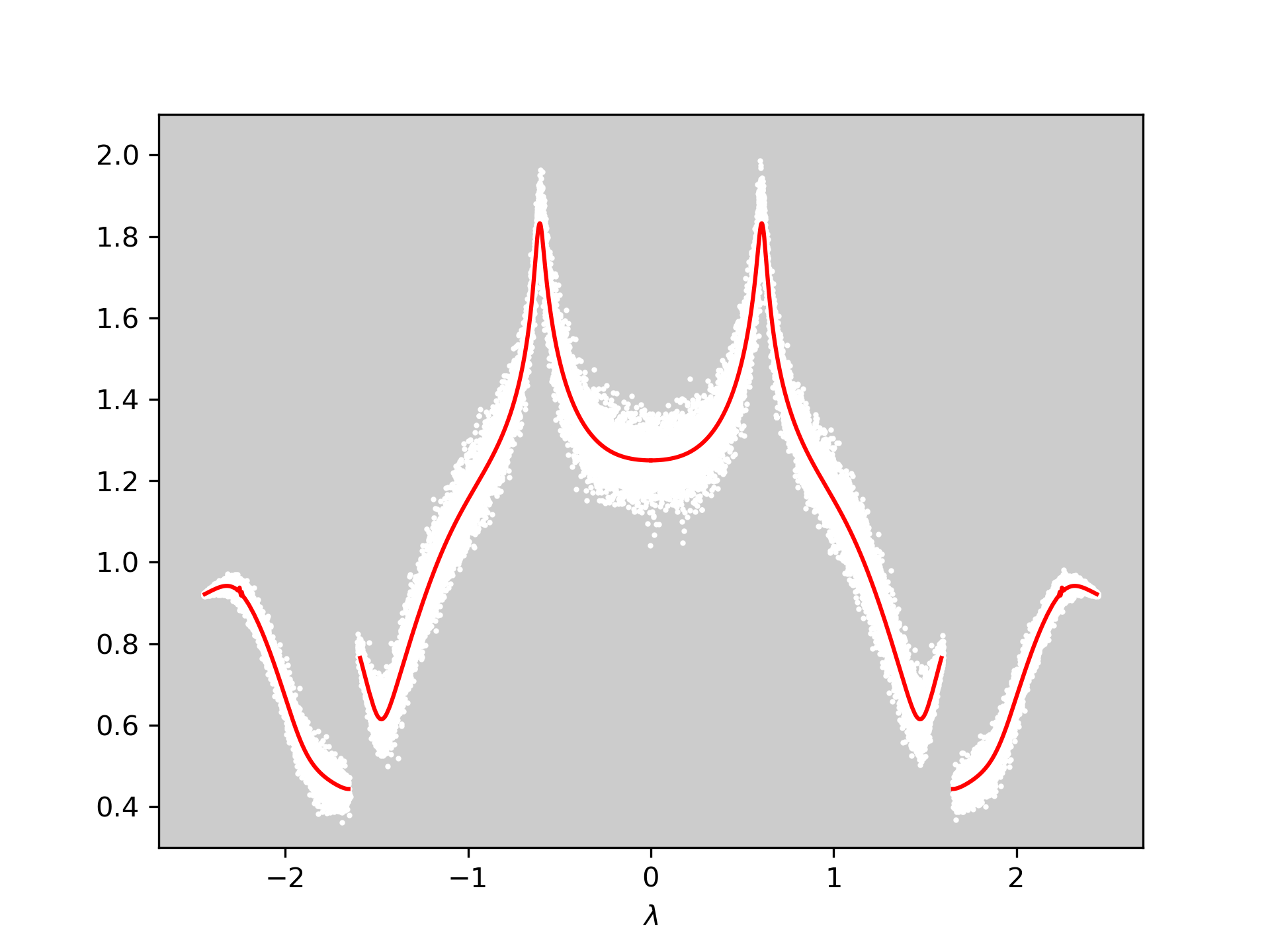}
        \caption{Sum of Coord. over $V_1$}
        \label{sub:house_1}
    \end{subfigure}
    \hfill
    \begin{subfigure}[b]{.3\textwidth}
        \centering
        \includegraphics[width=\textwidth]{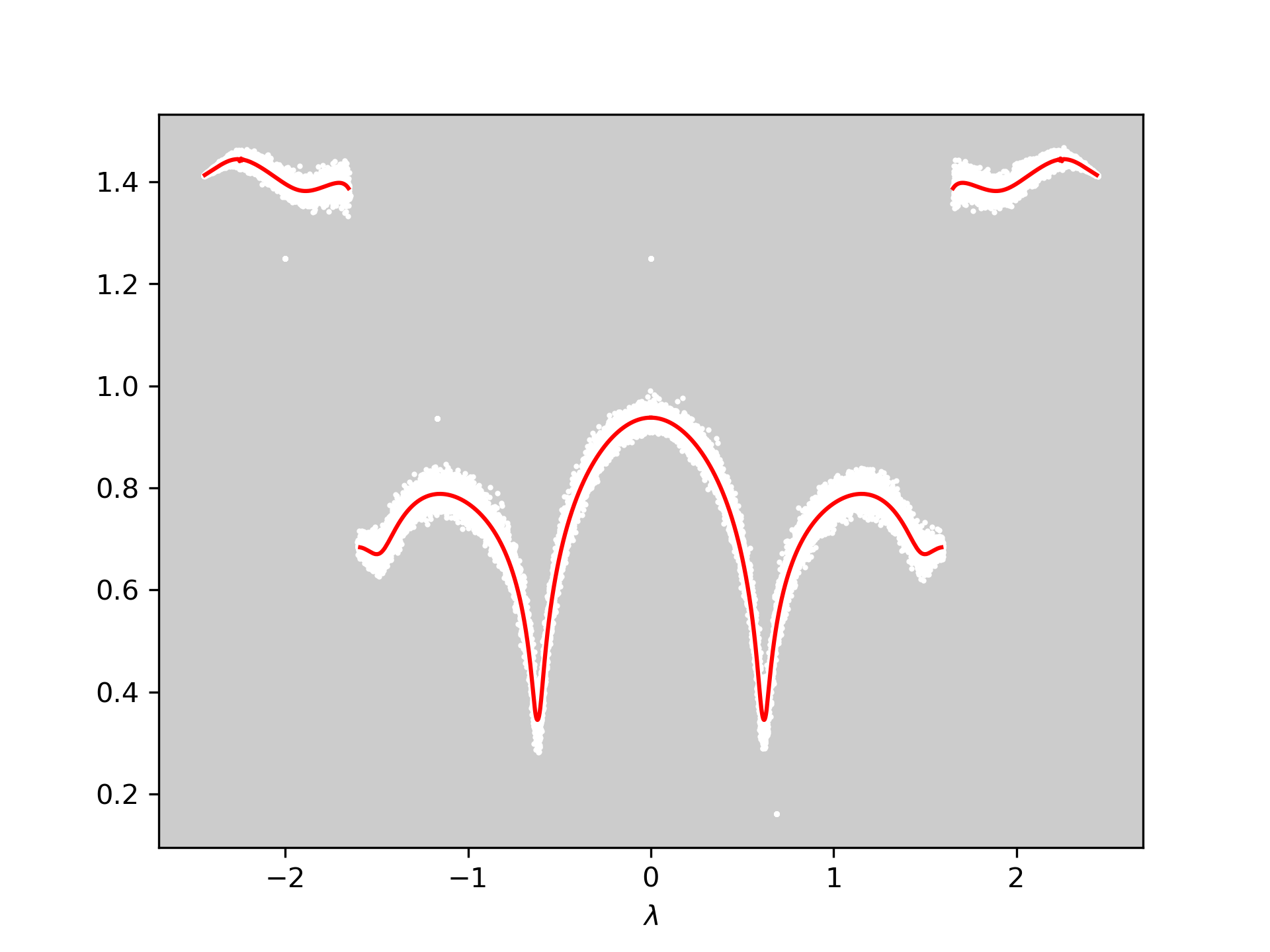}
        \caption{Sum of Coord. $V_2$ and $V_3$}
        \label{sub:house_2}
    \end{subfigure}
    \hfill
    \begin{subfigure}[b]{.3\textwidth}
        \centering
        \includegraphics[width=\textwidth]{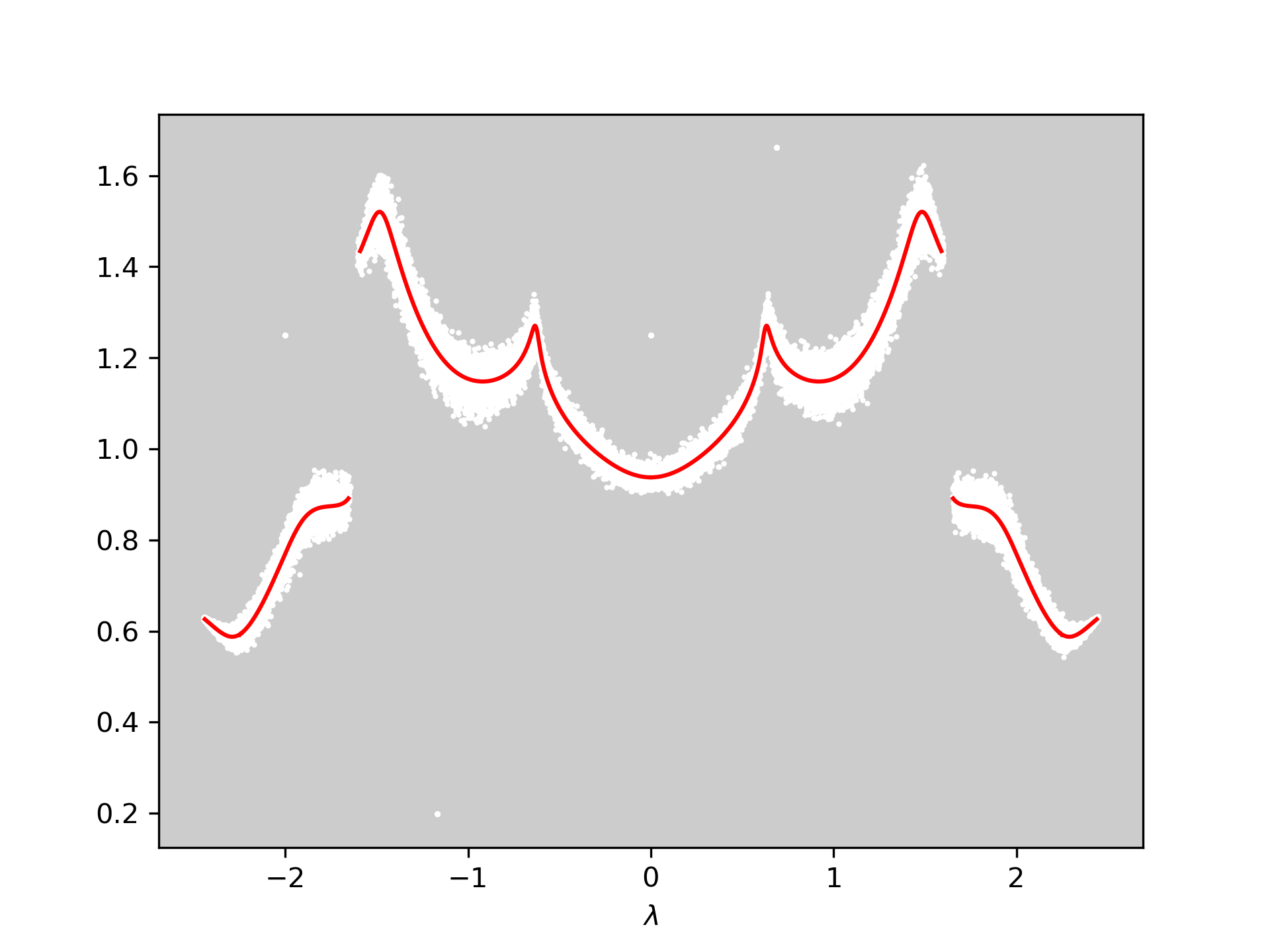}
        \caption{Sum of Coord. $V_4$ and $V_5$}
        \label{sub:house_3}
    \end{subfigure}
    \hfill
    \begin{subfigure}[b]{.45\textwidth}
        \centering
        \includegraphics[width=\textwidth]{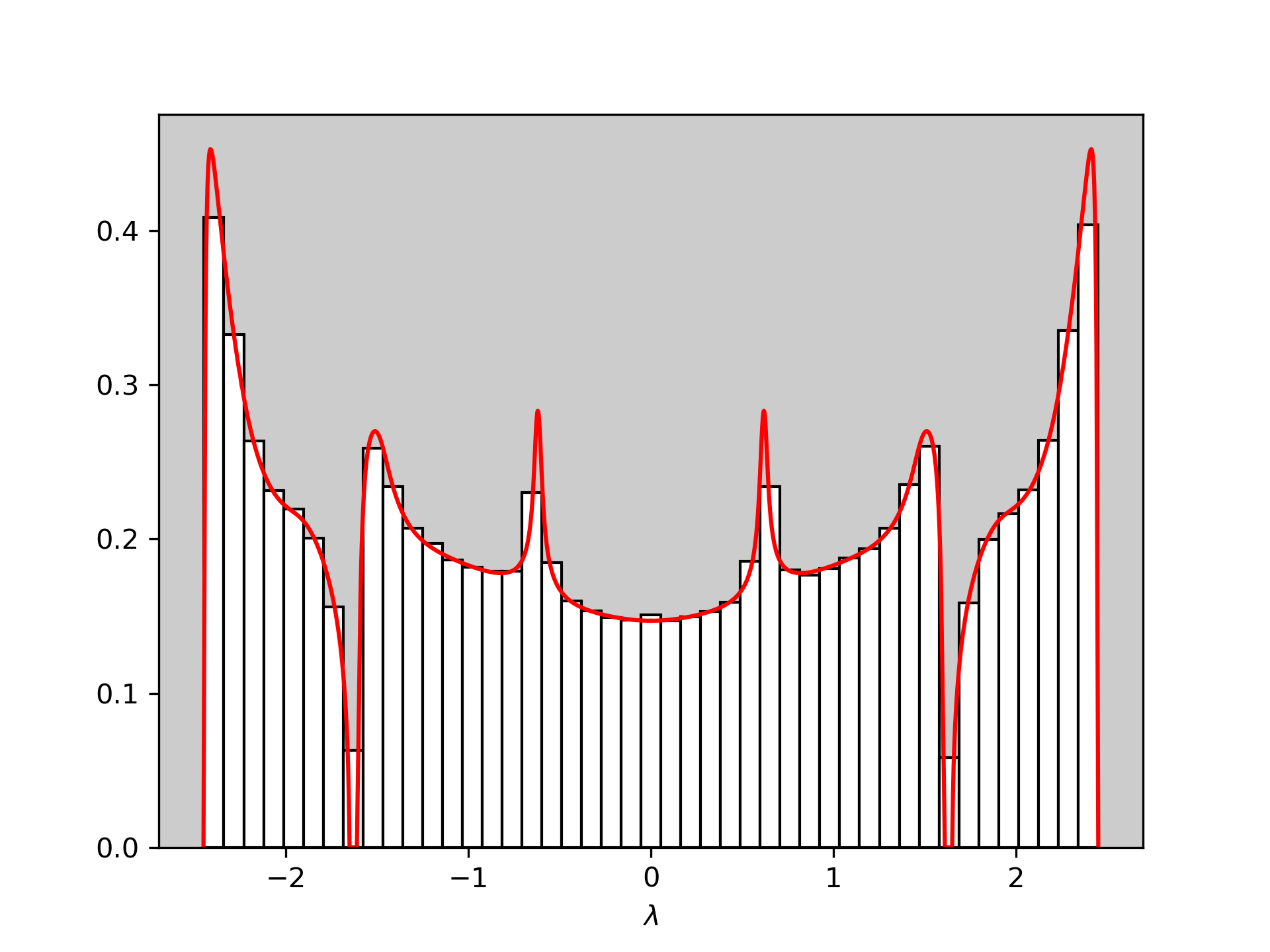}
        \caption{Spectral Density}
        \label{sub:house_spec}
    \end{subfigure}
    \caption{Average Eigenvector Coordinate over Partition Cells and Spectral Density}
    \label{fig:House}
\end{figure}
\end{example}

\section{Eigenvalues of the Adjacency Matrix}
\label{sec:eig}

The eigenvalues of an $S$-regular graph behave very similarly to the eigenvalues of a $d$-regular graph. For a $d$-regular graph the all-ones vector $\mathbf{1}$ is an eigenvector with eigenvalue $d$. From the perspective of $S$-regularity, this is due to the fact that the one-dimensional vector 1 is an eigenvector of the $1 \times 1$ matrix $d$, whose eigenvalue is $d$. In the following proposition we generalize this intuition showing that the eigenvalues of $S$ appear as eigenvalues of the adjacency matrix of an $S$-regular graph.

For $d$-regular graphs, $d$ is the largest eigenvalue of the adjacency matrix. Again, from the perspective of $S$-regularity this is because $d$ is the largest eigenvalue of the $1 \times 1$ matrix $d$. We generalize this fact by showing that if $\lambda_S$ is the largest eigenvalue of $S$, then $\lambda_S$ is also the largest eigenvalue of the adjacency matrix of an $S$-regular graph.

\begin{proposition}
Let $\lambda_S$ be the largest eigenvalue of $S$ and let $G$ be an $S$-regular graph. The largest eigenvalue of the adjacency matrix $A$ is $\lambda_S$.
\end{proposition}
\begin{proof}
By the Perron-Frobenius Theorem $\lambda_S$ is a positive real number with positive eigenvector $\phi_S$. By \Cref{prop:s-eigen} $\phi_S$ extends to a positive eigenvector $\phi$ of $A$ with eigenvalue $\lambda_S$ such that $\phi(v) = \phi_S(\tau(v))$.
Assume $\lambda > \lambda_S$ is an eigenvalue of $A$. Again, by Perron-Frobenius there exists a positive left eigenvector $\psi$ with value $\lambda$.
Now, $\lambda \psi^T \phi = \psi^T A \phi = \lambda_S \psi^T \phi$. This is a contradiction since $\psi^T \phi \neq 0$ by positivity, and $\lambda > \lambda_S$.
\end{proof}

\Cref{prop:s-eigen} characterized $k$ of the $n$ eigenvectors and eigenvalues of the adjacency matrix of an $S$-regular graph. We call the $k$ eigenvectors $\phi_1,\dots,\phi_k$ the \emph{$S$-eigenvectors} and their corresponding eigenvalues $\lambda_{S_1},\dots,\lambda_{S_k}$ the \emph{$S$-eigenvalues}.
We call the remaining $n-k$ eigenvectors $\phi_{k+1}, \dots, \phi_n$ the \emph{bulk eigenvectors} and their corresponding eigenvalues the \emph{bulk eigenvalues}. Note that our ordering $\phi_1,\dots,\phi_k,\phi_{k+1},\dots,\phi_n$ does not respect the magnitude of the eigenvalues. A bulk eigenvalue may be larger than an $S$-eigenvalue. By $\lambda_S$ we denote the largest $S$-eigenvalue, and by $\lambda_B$ we denote the largest magnitude of a bulk eigenvalue. 
In a $d$-regular graph the bulk eigenvectors are orthogonal to $\mathbf{1}$ since $\mathbf{1}$ is the eigenvector corresponding to $d$. \Cref{prop:s-eigen} shows the bulk eigenvectors are orthogonal to $\mathbf{1}$ as well as the indicator vectors $\one_{V_i}$ for each cell in the equitable partition. Note that for $d$-regular graphs $\one$ is the indicator vector for the single cell in the equitable partition given by $d$.

We put our observations on $S$-regular graphs to work by proving an upper bound on the largest eigenvalue of the adjacency matrix of a graph obtained by removing a subset of vertices from an $S$-regular graph. This theorem generalizes an analogous upper bound for $d$-regular graphs found in~\cite{10.5555/3002498}.

{\begin{theorem}\label{thm:eigen_remove}
Let $G$ be an $S$-regular graph and $C\subset V$. Further let $A_{\bar{C}}$ be the adjacency matrix of the subgraph induced by the complement of $C$. The maximum eigenvalue of $A_{\bar{C}}$ is bounded above by \[\lambda_S \left(1-\min_{1\leq i\leq k}\{c_i\}\right) + \lambda_B \left(\min_{1\leq i\leq k}\{c_i\}\right).\]
\end{theorem}
\begin{proof}
By \Cref{prop:s-eigen} there is a basis $\one_{V_i},\dots,\one_{V_k},\phi_{k+1},\dots,\phi_n$ where $\one_{V_i}$ is the indicator vector for the $i$th cell in the equitable partition, scaled by $1/\sqrt{n_i}$, and $\phi_i$ is a bulk eigenvector. Note that this is indeed a basis since the $S$-eigenvectors can be written as linear combinations of the indicator vectors. Moreover, by appropriately scaling the bulk eigenvectors it is an orthonormal basis.
Let $x$ be a unit vector such that $x(v) = 0$ for $v \in C$, { $x_i = \sum_{v\in V_i}x(v)^2$}, and write it as a linear combination $x = \sum_{i=1}^k \alpha_i \one_{V_i} + \sum_{i=k+1}^n \alpha_i \phi_i$.
By taking the inner product of $x$ and $\one_{V_i}$ for $1\leq i\leq k$ we compute
\[
\alpha_i = \sum_{v \in (V \setminus C)_i} \frac{x(v)}{\sqrt{n_i}} 
\leq \left(\sum_{v \in (V \setminus C)_i} x(v)^2 \right)^{1/2} \left( \frac{(1 - c_i)n_i}{n_i} \right)^{1/2} 
= \sqrt{x_i}\sqrt{1 - c_i}.
\]
{and so,
\[
\sum_{i=1}^k \alpha_i^2 \leq \sum_{i=1}^k x_i(1-c_i) \leq \left(1-\min_{1\leq i \leq k}\{c_i\}\right) \sum_{i=1}^k x_i = 1 - \min_{1\leq i \leq k}\{c_i\}.
\]}
Finally, we bound the largest eigenvalue of $A_{\bar{C}}$ by
\begin{align*}
x^T A x &= \sum_{i=1}^k \alpha_i^2\one_{V_i}^T A \one_{V_i} + \sum_{i=k+1}^n \lambda_i \alpha_i^2 \\
&\leq \lambda_S \sum_{i=1}^k \alpha_i^2 + \lambda_B \sum_{i=k+1}^n \alpha_i^2 \\
&\leq \lambda_S \sum_{i=1}^k \alpha_i^2 + \lambda_B \left(1 - \sum_{i=1}^k \alpha_i^2\right) \\
&\leq \lambda_S \left(1-\min_{1\leq i\leq k}\{c_i\}\right) + \lambda_B \left(\min_{1\leq i\leq k}\{c_i\}\right),
\end{align*}
since $\lambda_S \geq \lambda_B$.
\end{proof}}

Note that for a $d$-regular graph the bound reduces to $(1-c)d + c\lambda_B$, for $|C| = cn$, which matches the bound obtained in~\cite{10.5555/3002498}.
An immediate corollary of \Cref{thm:eigen_remove} provides an upper bound on the number of walks in $G$ that avoid the vertices in $C$.

\begin{corollary}
Let $G$ be an $S$-regular graph and $C\subset V$ and $m=|V|-|C|$. The total number of walks of length $\ell$ that avoid $C$ is bounded above by $$m\left(\lambda_S \left(1-\min_{1\leq i\leq k}\{c_i\}\right) + \lambda_B \left(\min_{1\leq i\leq k}\{c_i\}\right)\right)^\ell .$$
\end{corollary}
\begin{proof}
Let $\psi_1,\dots,\psi_m$ be an orthonormal basis of eigenvectors for $A_{\bar{C}}$ with eigenvalues $\gamma_1 \geq\dots \geq \gamma_m$.
The total number of walks of length $\ell$ in $G$ avoiding $C$ is given by $\one^T A_{\bar{C}}^\ell \mathbf{1}$. By writing $\mathbf{1} = \sum_{i=1}^m \beta_i \gamma_i$ as a linear combination of eigenvectors we obtain 
\begin{align*}
\one^T A_{\bar{C}}^\ell \mathbf{1} &= \sum_{i=1}^m \beta_i^2 \gamma_i^\ell\\ 
&\leq \gamma_1^\ell \sum_{i = 1}^m \beta_i^2\\
&= \gamma_1^\ell m\\
&\leq m\left(\lambda_S \left(1-\min_{1\leq i\leq k}\{c_i\}\right) + \lambda_B \left(\min_{1\leq i\leq k}\{c_i\}\right)\right)^\ell .
\end{align*}
\end{proof}

\subsection{Expander Mixing Lemma}
The expander mixing lemma is a statement about the second largest eigenvalue of a $d$-regular graph. Intuitively, the expander mixing lemma says that for any two subsets $B, C \subseteq V$ the difference between the actual number of edges between $B$ and $C$ and the expected number of edges between $B$ and $C$ depends on the second largest eigenvalue of the adjacency matrix. In a sense, it tells us that the larger the spectral gap the more a $d$-regular graph looks like a random $d$-regular graph. For more details on the expander mixing lemma in $d$-regular graphs see~\cite{Hoory2006}. Variants of the expander mixing lemma have been proven for biregular graphs~\cite{DeWinter2012}, and matrix-weighted graphs~\cite{Hansen2021}.
In this section we show three bounds for the expander mixing lemma in $S$-regular graphs, each depending on different parameters. The proof of each bound showcases a different generalization of a proof technique from $d$-regular graphs to $S$-regular graphs. In the $S$-regular expander mixing theorem the largest bulk eigenvalue $\lambda_B$ appears in the bound. This is expected since the second largest eigenvalue of a $d$-regular graph is the largest bulk eigenvalue when viewing the $d$-regular graph as an $S$-regular graph.

Before we state the expander mixing lemma for $S$-regular graphs we need the expected number of edges two sets of vertices in a random $S$-regular graph. The following lemma tells us this expected value.

\begin{lemma}
Let $G$ be a random $S$-regular graph and let $B,C \subseteq V$ be two sets of vertices. The expected number of edges between $B$ and $C$ is \[\mathbb{E}[|E(B,C)|] = \sum_{i=1}^k \sum_{j=1}^k \sqrt{\frac{s_{ij} s_{ji}}{n_i n_j}}|B_i||C_j|. \]
\end{lemma}
\begin{proof}
Consider the two subsets $B_i = B \cap V_i$ and $C_j = C \cap V_j$ and the bipartite graph obtained by the vertices $B_i \cup C_j \subseteq V$ and the edges $E(B_i, C_j) \subseteq E$.
Now for $u \in B_i$ and $v \in C_j$ the probability of an edge between them is
\[ Pr[(u, v) \in E(B_i, C_j)] = \frac{s_{ij}}{n_j} = \frac{s_{ji}}{n_i} = \sqrt{\frac{s_{ij}s_{ji}}{n_i n_j}} \] where the equality follows from the balance equation. By linearity of expectation we have \[\mathbb{E}[|E(B_i, C_i)|] = \sqrt{\frac{s_{ij}s_{ji}}{n_i n_j}}|B_i||C_j|.\]
Again by linearity of expectation, summing over all cells in the equitable partition we obtain \[\mathbb{E}[|E(B,C)|] = \sum_{i=1}^k \sum_{j=1}^k \sqrt{\frac{s_{ij} s_{ji}}{n_i n_j}}|B_i||C_j|. \]
\end{proof}

Note that when $G$ is a $d$-regular graph our expression reduces to $\mathbb{E}[|E(B,C)|] = \frac{d}{n}|B||C|$ which is the expected number of edges between two sets of vertices in a random $d$-regular graph.
Our next theorem is related to the expander mixing lemma. For a subset $B$ with $|B_i| = b_in_i$ it provides a bound on a quantity related to $\left| |N_{B_i}(v)| - b_i s_{\tau(v)i} \right|$ which is the difference between the actual number of neighbors $v$ has in $B_i$ and the expected number of neighbors $v$ has in $B_i$. As a corollary we obtain a variant of the expander mixing lemma. An analogous result for $d$-regular graphs can be found in~\cite{10.5555/3002498}.

\begin{theorem}\label{thm:eigen}
For an $S$-regular graph $G$ and a set of vertices $B \subseteq V$ we have the inequality \[ \sum_{v \in V} \sum_{i=1}^k \left( \left| N_{B_i}(v) \right| - b_i s_{\tau(v)i} \right)^2 \leq \lambda_B^2 \sum_{i=1}^k b_i(1 - b_i)n_i.\]
\end{theorem}
\begin{proof}
Define the vector $x$ such that $x(v) = 1 - b_i$ if $v \in B_i$ and $x(v) = -b_i$ if $v \in (V \setminus B)_i$. Since $x$ is orthogonal to the $S$-eigenvectors we have $\langle Ax, Ax \rangle \leq \lambda^2_B \langle x, x \rangle$.
The right hand side is equal to \[\lambda^2_B \langle x, x \rangle = \lambda^2_B \sum_{i=1}^k b_i n_i (1 - b_i)^2 + (1 - b_i)n_ib_i^2 = \lambda_B^2 \sum_{i=1}^k b_i(1-b_i)n_i,\] and the left hand side is equal to \[\langle Ax, Ax \rangle = \sum_{v\in V} \sum_{i=1}^k \left( \left(1 - b_i \right) \left| N_{B_i}(v) \right| - b_i \left( s_{c(v),i} - N_{B_i}(v) \right)\right)^2 = \sum_{v\in V} \sum_{i=1}^k \left( \left| N_{B_i}(v) \right| - b_i s_{c(v), i} \right)^2. \]
\end{proof}

\begin{corollary}
For an $S$-regular graph with sets of vertices $B, C \subseteq V$ we have the inequality \[ \left| \left|E(B, C)\right| - \sum_{i=1}^k \sum_{j=1}^k \sqrt{\frac{s_{ij} s_{ji}}{n_i n_j}}|B_i||C_j| \right| \leq \lambda_B n \left(b \sum_{i=1}^k c_i(1-c_i) \right)^{1/2}.\]
\end{corollary}
\begin{proof}
By \Cref{thm:eigen} we have \[\sum_{v \in B} \sum_{i=1}^k \left( \left| N_{B_i}(v) \right| - b_i s_{c(v),i} \right)^2 \leq \sum_{v \in V} \sum_{i=1}^k \left( \left| N_{B_i}(v) \right| - b_i s_{c(v),i} \right)^2 \leq \lambda^2 \sum_{i=1}^k b_i(1 - b_i)n_i\] for any $B \subset V$.
By Cauchy-Schwarz we have
\begin{align*}
\left| |E(B, C)| - \sum_{i=1}^k \sum_{j=1}^k \sqrt{\frac{s_{ij} s_{ji}}{n_i n_j}}|B_i||C_j| \right| & =
\left| |E(B, C)| - \sum_{i=1}^k \sum_{j=1}^k n_i b_i s_{ij} c_j \right| \\
&\leq \sum_{v \in B} \sum_{i = 1}^k \left| |N_{C_i}(v)| - c_j s_{c(v),i} \right| \\
&\leq \sqrt{bn} \left( \sum_{v \in B}\left( \sum_{i = 1}^k |N_{C_i}(v)| - c_j s_{c(v),i} \right)^2 \right)^{1/2}\\
&\leq \sqrt{bn} \left( \lambda^2 \sum_{i=1}^k c_i(1-c_i) n_i \right)^{1/2}\\
&\leq \lambda n \left( b \sum_{i=1}^k c_i(1-c_i) \right)^{1/2}.
\end{align*}
\end{proof}

Our next proof of the expander mixing lemma is a generalization of the most well-known version of the lemma for $d$-regular graphs.

\begin{theorem}\label{thm:eml-original}
For an $S$-regular graph $G$ and sets of vertices $B, C \subseteq V$ we have the inequality \[ \left| |E(B, C)| - \sum_{i=1}^k \sum_{j=1}^k \sqrt{\frac{s_{ij} s_{ji}}{n_i n_j}}|B_i| |C_j| \right| \leq \lambda_B \sqrt{|B||C|}. \]
\end{theorem}
\begin{proof}
Let $J$ be the $n \times n$ block matrix such that the $(i,j)$-block is an $n_i \times n_j$ matrix with every entry set to $\frac{s_{ij}}{n_j}$. Since $s_{ij} n_i = s_{ji} n_j$ we have that $J$ is symmetric.
Moreover, we have the equality \[\left| |E(B, C)| - \sum_{i=1}^k \sum_{j=1}^k \sqrt{\frac{s_{ij} s_{ji}}{n_i n_j}} |B_i| |C_j| \right| = \left| \mathbf{1}_B^T \left( A - J \right) \mathbf{1}_C \right| \] by combining \[\mathbf{1}_B^T A \mathbf{1}_C = \left| E(B, C) \right|\] and \[ \mathbf{1}^T_B J \mathbf{1}_C = \sum_{i=1}^k \sum_{j=1}^k \sqrt{\frac{s_{ij} s_{ji}}{n_i n_j}} |B_i| |C_j|. \]

By our construction of $J$ each $S$-eigenvector $\phi_i$ is an eigenvector of $J$ with eigenvalue $\lambda_{S_i}$ since $J\phi_i(v) = S\phi_{S_i}(\tau(v)) = \lambda_{S_i} \phi_{S_i}(\tau(v)) = \lambda_{S_i} \phi_i(v).$
Moreover, each bulk eigenvector $\phi_i$ is in the kernel of $J$. To see this, note that $J$ has rank $k$ since it has $k$ unique columns and the $S$-eigenvectors $\phi_1,\dots,\phi_k$ are linearly independent. Hence, the column space of $J$ is spanned by the $S$-eigenvectors, but the bulk eigenvectors are orthogonal to the $S$-eigenvectors, so they must be in the kernel of $J$ since $J$ is symmetric.
It follows that $A$ and $J$ have the same eigenvectors which are the eigenvectors of $A - J$. The $S$-eigenvectors are eigenvectors of $A - J$ with eigenvalue zero, and each bulk eigenvector $\phi_i$ is an eigenvector of $A- J$ with eigenvalue $\lambda_i$.

To conclude the proof we upper bound $|\mathbf{1}_B^T (A - J) \mathbf{1}_C|$. By Cauchy-Schwarz we have \[\left| \mathbf{1}_B^T (A - J) \mathbf{1}_C \right| = \left| \langle \mathbf{1}_B, \left( A-J \right)\mathbf{1}_C \rangle \right| \leq \|\mathbf{1}_B\| \|(A-J)\mathbf{1}_C \|. \]
Next, by writing $\mathbf{1}_C$ as a linear combination of eigenvectors and using the fact that $(A-J)$ is symmetric we obtain the inequality
\begin{align*}
\| (A-J) \mathbf{1}_C \|^2 &= \langle (A-J) \mathbf{1}_C, (A-J) \mathbf{1}_C \rangle \\
&= \langle \mathbf{1}_C, (A - J)^2 \mathbf{1}_C \rangle \\
&= \left\langle \mathbf{1}_C, \sum_{i=k+1}^{n} (A-J)^2 \left\langle \mathbf{1}_C, \phi_i \right\rangle \phi_i \right\rangle\\
&= \sum_{i=k+1}^{n} \lambda_i^2 \left\langle \mathbf{1}_C, \phi_i \right\rangle^2\\
&\leq \lambda^2_B \|\mathbf{1}_C\|^2.
\end{align*}
Putting this all together we get $\|(A-J) \mathbf{1}_C\| \leq \lambda \|\mathbf{1}_C\|$ implying that \[ \left| |E(B,C)| - \sum_{i=1}^k \sum_{j=1}^k \frac{\sqrt{s_{ij} s_{ji}}}{\sqrt{n_i n_j}} |B_i| |C_j|\right| \leq \lambda_B \|\mathbf{1}_B\| \|\mathbf{1}_C\| = \lambda_B \sqrt{|B||C|}. \]
\end{proof}

By generalizing another well-known proof of the expander mixing lemma we obtain a slightly better bound. The bound is tighter since it includes fractional terms in the form $1 - b_i = 1 - |B_i|/n_i$ as weights on the sizes of the sets.

\begin{theorem}
For an $S$-regular graph $G$ and sets of vertices $B, C \subseteq V$ we have the inequality \[ \left| |E(B, C)| - \sum_{i=1}^k \sum_{j=1}^k \sqrt{\frac{s_{ij} s_{ji}}{n_i n_j}}|B_i| |C_j| \right| \leq \lambda_B \sqrt{\sum_{i=1}^k \sum_{j=1}^k|B_i||C_j| \left(1 - b_i \right) \left( 1 - c_j \right)}. \]
\end{theorem}
\begin{proof}
Consider the vectors $\tilde{\mathbf{1}}_B = \mathbf{1}_B - \sum_{i=1}^k b_i \mathbf{1}_{V_i}$ and $\tilde{\mathbf{1}}_C = \mathbf{1}_C - \sum_{i=1}^k c_i \mathbf{1}_{V_i}$ which are both orthogonal to the $S$-eigenvectors since
\[ \mathbf{1}_B^T \phi_j = \sum_{i=1}^k |B_i|\phi_{S_j}(i) = \left( \sum_{i=1}^k b_i \mathbf{1}_{V_i} \right)^T \phi_j . \]

We have

\begin{align*}
\one_B^T A \one_C &= \left(\tilde{\mathbf{1}}_B + \sum_{i=1}^k b_i \mathbf{1}_{V_i}\right)^T A \left( \tilde{\mathbf{1}}_C + \sum_{i=1}^k c_i \mathbf{1}_{V_i}\right)\\
&= \left(\sum_{i=1}^k b_i \mathbf{1}_{V_i}\right)^T A \left(\sum_{i=1}^k c_i \mathbf{1}_{V_i} \right) + \left(\mathbf{1}_B - \sum_{i=1}^k b_i \mathbf{1}_{V_i} \right)^T A \left(\mathbf{1}_C - \sum_{i=1}^k c_i \mathbf{1}_{V_i}\right),
\end{align*}
since the terms $\tilde{\one}_B^T A \left( \sum_{i=1}^k c_i \one_{V_i} \right)$ and $\left( \sum_{i=1}^k b_i \one_{V_i} \right)^T A \tilde{\one}_C$ are zero. This is because $\tilde{\one}_B$ and $\tilde{\one}_C$ are orthogonal to the $S$-eigenvectors and the indicator vectors $\one_{V_i}$ are in the span of the $S$-eigenvectors.
Moreover, we have
\[\left( \sum_{i=1}^k b_i \mathbf{1}_{V_i} \right)^T A \left( \sum_{i=1}^k c_i \mathbf{1}_{V_i} \right) = \sum_{i=1}^k \sum_{j=1}^k \sqrt{\frac{s_{ij}s_{ji}}{n_i n_j}}|B_i||C_j|\] which is the expected number of edges between $B$ and $C$.
Putting these together we obtain
\begin{align*}
\left| |E(B,C)| - \sum_{i=1}^k \sum_{j=1}^k \sqrt{\frac{s_{ij}s_{ji}}{n_i n_j}}|B_i||C_j| \right| &\leq \left(\mathbf{1}_X - \sum_{i=1}^k b_i \mathbf{1}_{V_i} \right)^T A \left(\mathbf{1}_C - \sum_{i=1}^k c_i \mathbf{1}_{V_i}\right)\\
&\leq \left\|\mathbf{1}_B - \sum_{i=1}^k b_i \mathbf{1}_{V_i} \right\| \left\| A \left( \mathbf{1}_C - \sum_{i=1}^k c_i \mathbf{1}_{V_i} \right) \right\|.
\end{align*}
By the same argument as in the proof of \Cref{thm:eml-original} we obtain \[ \left\|A \left( \mathbf{1}_B - \sum_{i=1}^k b_i \mathbf{1}_{V_i} \right) \right\| \leq \lambda_B \left\|\mathbf{1}_C - \sum_{i=1}^k c_i \mathbf{1}_{V_i} \right\|. \]
Using the fact that \[\left\|\mathbf{1}_B - \sum_{i=1}^k b_i \mathbf{1}_{V_i} \right\| = \sqrt{\sum_{i=1}^k \left| B_i \right| \left( 1 - b_i \right)}\] we obtain the final result \[ \left| |E(B,C)| - \sum_{i=1}^k \sum_{j=1}^k \sqrt{\frac{s_{ij}s_{ji}}{n_i n_j}}|B_i||C_j| \right| \leq \lambda_B \sqrt{\sum_{i=1}^k \sum_{j=1}^k|B_i||C_j| \left(1 - b_i \right) \left( 1 - c_j \right)}. \]
\end{proof}

To wrap up the section, we show that the converse to the expander mixing lemma due to Bilu and Linial~\cite{Bilu2006} also implies a converse to the expander mixer lemma for $S$-regular graphs. Note the difference in assumptions. The converse assumes the two sets $B$ and $C$ are disjoint, but the statement of the expander mixing lemma allows them to have non-empty intersection.

\begin{theorem}
If $G$ is an $S$-regular graph such that for all $B, C \subset V$ with $B \cap C = \emptyset$ the inequality \[ \left| |E(B, C)| - \sum_{i=1}^k \sum_{j=1}^k \sqrt{\frac{s_{ij} s_{ji}}{n_i n_j}} |B_i| |C_j| \right| \leq \alpha \sqrt{|B||C|}\] holds, then the largest bulk eigenvalue is bounded above by $O \left ( \alpha + (\log(d_{\textrm{max}}/\alpha) + 1) \right)$ where $d_{\textrm{max}}$ is the maximum degree of $G$.
\end{theorem}
\begin{proof}
Our proof is nearly identical to the proof of Corollary 5.1 of \cite{Bilu2006}.
The largest eigenvalue of $A - P$ is the largest bulk eigenvalue of $A$. Moreover, $A - P$ is symmetric and the largest $\ell_1$-norm of a row is at most $2d_{\textrm{max}}$. By Lemma 3.3 of \cite{Bilu2006} it suffices to show that $\| \mathbf{1}_B^T(A - P)\mathbf{1}_C \| \leq \alpha \|\mathbf{1}_B\|\|\mathbf{1}_C\|$, but this follows directly from our assumptions on $B$ and $C$.
\end{proof}

\subsection{Alon-Boppana Bound}
The Alon-Boppana bound provides a lower bound on the second largest eigenvalue of a $d$-regular graph. In this section we prove an analogous lower bound on the largest bulk eigenvalue of an $S$-regular graph. One variant of the Alon-Boppana bound says that for a $d$-regular graph $\lambda_2 > 2 \sqrt{d - 1} - o(1)$ where the term $o(1)$ tends to zero as the size of the graph tends towards infinity, and this is the variant we generalize. We recommend \cite{Hoory2006} for a survey on the bound.
For $d$-regular graphs the bound is computed by counting the number of closed walks on the $d$-regular tree, which is the universal cover of a $d$-regular graph.
Similarly, we will begin by counting the number of closed walks on the $S$-regular tree, which is the universal cover of an $S$-regular graph. Note that the length of any closed walk on a tree must be even.
In the following lemma we provide a lower bound on $W_i^{(2\ell)}$ in terms of the eigenbasis of $S$.

\begin{lemma}\label{lem:treewalks}
The number of closed walks of length $2\ell$ on the $S$-regular tree rooted at a vertex in $V_i$ is bounded below by \[W_i^{(2\ell)} > \sum_{j=1}^k \lambda_{S_j}^\ell \left\langle \one, \phi_{S_j} \right\rangle \phi_{S_j}(i). \]
\end{lemma}
\begin{proof}
Consider the following recurrence relation:
\begin{align*}
\tilde{W}_i^{(\ell)} &= \sum_{j=1}^k s_{ij} \tilde{W}_j^{(\ell - 1)}\\
\tilde{W}_i^{(0)} &= 1.
\end{align*}
This recurrence counts, not necessarily closed, walks of length $\ell$. These become closed walks of length $2\ell$ by following with the same walk but in the opposite direction. However, $\tilde{W}_i^{(\ell)}$ does not count all possible closed walks of length $2\ell$, only walks of a special type.
The closed walks counted by $\tilde{W}_i^{(\ell)}$ are symmetric in the sense that the $j^{\textrm{th}}$ vertex in the walk of length $\ell$ appears as both the $j^{\textrm{th}}$ and $(2\ell - j)^\textrm{th}$ vertex in its corresponding closed walk of length $2\ell$. Hence, $W_i^{(2\ell)} > \tilde{W}_i^{(\ell)}$.

The recurrence relation can be written in matrix form as $\tilde{W}_i^{(\ell)} = (S^\ell \one)(i)$ and by writing $\one$ as a linear combination of the eigenvectors of $S$ we obtain \[W_i^{(2\ell)} > \tilde{W}_i^{(\ell)} = (S^\ell \one)(i) = \sum_{j=1}^k \lambda_{S_j}^\ell \left\langle \one, \phi_{S_j} \right\rangle \phi_{S_j}(i).\]
\end{proof}

We now present an Alon-Boppana-like bound for $S$-regular graphs.

\begin{theorem}
For an $S$-regular graph $G$ we have the inequality $\lambda_B > \sqrt{\lambda_S} - o(1)$.
\end{theorem}
\begin{proof}
We begin by noting that the number of closed walks of length $2\ell$ in $G$ is given by 
\begin{align*}
\Tr A^{2\ell} &= \sum_{i=1}^k \lambda_{S_i}^{2\ell} + \sum_{i=k+1}^{n} \lambda_{i}^{2\ell}\\ 
&\leq \sum_{i=1}^k \lambda_{S_i}^{2\ell} + n \lambda_B^{2\ell}
\end{align*}
The number of closed walks of length $2\ell$ rooted at a vertex $v \in V_i$ is bounded below by the number of closed walks rooted at an element of the fiber of $v$ in the universal cover of $G$. The universal cover of $G$ is the $S$-regular tree, so by summing over the cells of the equitable partition and applying \Cref{lem:treewalks} we have \[\sum_{i=1}^k \sum_{j=1}^k n_i \lambda_{S_j}^\ell \langle \one, \phi_{S_j} \rangle \phi_{S_j}(i) < \Tr A^{2\ell} \leq \sum_{i=1}^k \lambda_{S_i}^{2\ell} + n \lambda_B^{2\ell}.\]
By rearranging the inequality we obtain \[\lambda_B > \left( \frac{1}{n}\sum_{i=1}^k n_i \left( \sum_{j=1}^k \lambda_{S_j}^{\ell} \langle \mathbf{1}, \phi_{S_j} \rangle \phi_{S_j}(i) \right) - \frac{1}{n} \sum_{i=1}^k \lambda_{S_i}^{2\ell} \right)^{\frac{1}{2\ell}}. \]
We consider the limit as $n \rightarrow \infty$ while setting $\ell = o(\log n)$.
Clearly, we have \[\frac{1}{n} \sum_{i=1}^k \lambda_{S_i}^{2\ell} = o(1).\]
Recall that $n_i/n$ is fixed for fixed $S$, so all we need is an asymptotic bound on \[\left( \sum_{i=1}^k \sum_{j=1}^k \lambda_{S_j}^{\ell} \langle \mathbf{1}, \phi_{S_j} \rangle \phi_{S_j}(i) \right)^{\frac{1}{2\ell}}.\]
Set $m = \max_{i,j}\{ \langle \mathbf{1}, \phi_{S_j} \rangle \phi_{S_j}(i)$ then for fixed $p$ and $q$ we have \[\left( \lambda_S^{\ell} \langle \mathbf{1}, \phi_{S_p} \rangle \phi_{S_q}(p) \right)^\frac{1}{2\ell} \leq \left( \sum_{i=1}^k \sum_{j=1}^k \lambda_{S_j}^{\ell} \langle \mathbf{1}, \phi_{S_j} \rangle \phi_{S_j}(i) \right)^{\frac{1}{2\ell}} \leq \left( k^2 \lambda_S^{\ell}m \right)^\frac{1}{2\ell},\] so by the squeeze theorem we have \[ \lim_{\ell \rightarrow \infty} \left( \sum_{i=1}^k \sum_{j=1}^k \lambda_{S_j}^{\ell} \langle \mathbf{1}, \phi_{S_j} \rangle \phi_{S_j}(i) \right)^{\frac{1}{2\ell}} = \sqrt{\lambda_S}, \]
which yields the desired inequality \[\lambda_B > \sqrt{\lambda_S} - o(1).\]
\end{proof}

For $d$-regular graphs our bound reduces to $\lambda_2 > \sqrt{d} - o(1)$ which isn't quite as good as the actual bound of $\lambda_2 > 2\sqrt{d-1} - o(1)$. This is because \Cref{lem:treewalks} only obtains a lower bound on the number of closed walks in the $S$-regular tree. However, it is straightforward to count the number of closed walks in the $d$-regular tree exactly. We conjecture that an exact count of the number of closed walks in the $S$-regular tree would yield an Alon-Boppana bound for $S$-regular graphs that would match the $d$-regular case.

In our last theorem we use $\lambda_S$ and $\lambda_B$ to upper bound the diameter of an $S$-regular graph.
Our result generalizes Chung's bound for $d$-regular graphs~\cite{Chung1989}. 
\begin{theorem}
The diameter of an $S$-regular graph is bounded above by $O \left( {\log(n)}\right)$.
\end{theorem}
\begin{proof}
The diameter of a graph is the minimum value $m$ such that $(A^t)_{j\ell}>0$  for some $t\leq m$. Let $u,v\in V$ with $\tau(u)=i$ and $\tau(v)=j$. We have

\begin{align*}
A^m_{uv} &= \sum_{\ell=1}^k \lambda_{S_\ell}^m \left( \phi_\ell \phi_\ell^T \right)_{uv} + \sum_{\ell=k+1}^n \lambda_{\ell}^m \left( \phi_\ell \phi_\ell^T \right)_{uv} \\
&\geq \left|\sum_{\ell=1}^k \lambda_{S_\ell}^m \left( \phi_\ell \phi_\ell^T \right)_{uv}\right| - \left|\sum_{\ell=k+1}^n \lambda_{\ell}^m \left( \phi_\ell \phi_\ell^T \right)_{uv}\right| \\
&\geq \left|\sum_{\ell=1}^k \lambda_{S_\ell}^m \left( \phi_\ell \phi_\ell^T \right)_{uv}\right| - \sum_{\ell=k+1}^n \lambda_{\ell}^m \left( \phi_\ell \phi_\ell^T \right)_{uv}\\
&\geq \left|\sum_{\ell=1}^k \lambda_{S_\ell}^m \left( \phi_\ell \phi_\ell^T \right)_{uv}\right| - |\lambda_B^m| \left| \sum_{\ell=k+1}^n \left(\phi_\ell \phi_\ell^T \right)_{uv}  \right|\\
&\geq\left|\sum_{\ell=1}^k \lambda_{S_\ell}^m \left( \phi_\ell \phi_\ell^T \right)_{uv}\right| - |\lambda_B^m| \left| \sum_{\ell=k+1}^n \left| \phi_\ell(i) \right| \left| \phi_\ell(j) \right| \right| \\
&\geq \left|\sum_{\ell=1}^k \lambda_{S_\ell}^m \left( \phi_\ell \phi_\ell^T \right)_{uv}\right| - |\lambda_B^m| \left( \sum_{\ell=k+1}^n \phi_\ell(u)^2\right)^{1/2}\left( \sum_{\ell=k+1}^n \phi_\ell(v)^2\right)^{1/2} \\
&= \left|\sum_{\ell=1}^k \lambda_{S_\ell}^m \left( \phi_\ell \phi_\ell^T \right)_{uv}\right| - |\lambda_B^m| \left( 1 - \sum_{\ell=1}^k \phi_\ell(u)^2 \right)^{1/2} \left( 1 - \sum_{\ell=1}^k \phi_\ell(v)^2 \right)^{1/2}
\end{align*}
By \Cref{cor:s-eig}
\[
\left|\sum_{\ell=1}^k \lambda_{S_\ell}^m \left( \phi_\ell \phi_\ell^T \right)_{uv}\right| = \sqrt{\frac{(S^m)_{ij}(S^m)_{ji}}{n_in_j}} \text{ and } \sum_{\ell=1}^k \phi_\ell(u)^2 = \frac{1}{n_i}.
\]
Thus, when ${\sqrt{(S^m)_{ij}(S^m)_{ji}}}/{|\lambda_B^m|} > n-1>\sqrt{(n_i - 1)(n_j - 1)}$, we have $A^m_{uv}>0$. 
For large $m$, $(S^m)_{ij} = \Theta\left(\lambda_S^m\right)$, and so, $\sqrt{(S^m)_{ij}(S^m)_{ji}}/|\lambda_B^m| = \Theta\left(\left|\frac{\lambda_S}{\lambda_B}\right|^m\right)$. Because $\lambda_B$ is bounded below by $\sqrt{\lambda_S}-o(1)$ and bounded above by $\lambda_S$, the diameter is bounded above by $O\left({\log(n)}\right)$.
\end{proof}

\section{Acknowledgements}
This work was funded by the In-house Laboratory Independent Research and Naval Innovative Science and Engineering programs at the Naval Surface Warfare Center, Dahlgren Division.

\bibliography{S-Reg.bib}

\end{document}